\documentclass[11pt]{article}
\usepackage[left=1in, right=1in, top=1in]{geometry}
\usepackage{mathtools}
\usepackage[toc,page]{appendix}
\usepackage{amsmath}
\usepackage{amssymb}
\usepackage[normalem]{ulem}
\usepackage[utf8]{inputenc}
\usepackage{tikz}
\usepackage{xcolor}
\usepackage{tikz-cd}
\usepackage{amsfonts}
\usepackage{amsthm}
\usepackage{hyperref}
\usepackage{xhfill}
\usepackage{mathrsfs}
\usepackage[english]{babel}
\usepackage{amsmath}
\usepackage{relsize}
\usepackage{authblk}
\usepackage{lipsum}
\usepackage{xcolor}
\newcommand{\ibull}{I_{\bullet}}

\newcommand{\tanxt}{\mct_{X/T}}
\newcommand{\mctbull}{\mct_{\bullet}}
\newcommand{\dbull}{\dd_{\bullet}}

\newcommand{\tanxjtj}{\mct_{X_j/T_j}}

\newcommand{\sDef}{\on{sDef}}

\newcommand{\xbull}{X_{\bullet}}
\newcommand{\mclbull}{\mcl_{\bullet}}
\newcommand{\oxj}{\mc{O}_{X_j}}

\newcommand{\oxnot}{\mc{O}_{X_0}}

\newcommand{\mcebull}{\mce_{\bullet}}

\newcommand{\mco}{\mc{O}}
\newcommand{\mcy}{\mc{Y}}

\newcommand{\prij}{\on{pr}_{ij}}
\newcommand{\prj}{\on{pr}_j}

\newcommand{\homiisq}{\underline{\Hom}_R(\iisq,R)}
\newcommand{\homomrk}{\underline{\Hom}_R(\Omega_{B/k} \otimes R,R)}
\newcommand{\iisq}{I/I^2}
\newcommand{\nspoly}{k[x,y,\theta_1, \theta_2]}

\newcommand{\oo}{\mc{O}}

\newcommand{\mca}{\mc{A}^1}

\newcommand{\osig}{\mc{O}_{\Sigma}}

\newcommand{\ou}{\mc{O}_{U}}

\newcommand{\ovscg}{\ov{\mc{S}\mc{C}}_g}

\newcommand{\scg}{\mc{S}\mc{C}_g}

\newcommand{\ovt}{\overline{t}}

\newcommand{\pull}{{}^*}

\newcommand{\prjpull}{\on{pr}_j^*}

\newcommand{\oy}{\mc{O}_Y}

\newcommand{\mce}{\mc{E}}

\newcommand{\pr}{\on{pr}}
\newcommand{\coker}{\on{coker}}

\newcommand{\minus}{{}^{-}}

\newcommand{\bracket}{[ \ , \ ]}
\newcommand{\hra}{\hookrightarrow}

\newcommand{\plus}{{}^{+}}

\newcommand{\omegaxjtj}{\Omega_{X_j/T_j}^1}

\newcommand{\ovmfr}{\ov{\mfr}}

\newcommand{\duaxs}{\omega_{X_S/S}}
\newcommand{\oxs}{\mc{O}_{X_S}}

\newcommand{\duaxiti}{\omega_{X_i/T_i}}

\newcommand{\duaxjtj}{\omega_{X_j/T_j}}
\newcommand{\uniso}{\underline{\on{Isom}}}
\newcommand{\hh}{\mathbb{H}}

\newcommand{\otimestwo}{{}^{\otimes 2}}

\newcommand{\duaxt}{\omega_{X/T}}

\newcommand{\xovt}{X_{\ov{t}}}

\newcommand{\tildelt}{\tilde{\delta}}

\newcommand{\Hilb}{\on{Hilb}}

\newcommand{\phipull}{\phi^*}

\newcommand{\pripull}{\on{pr}_i^*}

\newcommand{\cc}{\mc{C}}

\newcommand{\mcl}{\mc{L}}

\newcommand{\pibos}{\pi_{\bos}}
\newcommand{\pipull}{\pi^*}

\newcommand{\ppull}{p^*}

\newcommand{\lra}[1]{\overset{#1}{\longrightarrow}}

\newcommand{\mcg}{\mc{G}}

\newcommand{\bos}{\on{bos}}

\newcommand{\prijpull}{\on{pr}_{ij}^*}
\newcommand{\prjipull}{\on{pr}_{ji}^*}

\newcommand{\xbos}{X_{\on{bos}}}

\newcommand{\omegaxt}{\Omega_{X/T}^1}
\newcommand{\tbos}{T_{\on{bos}}}

\newcommand{\oc}{\mc{O}_C}
\newcommand{\jj}{\mc{J}}

\newcommand{\ipull}{i^*}

\newcommand{\ot}{\mc{O}_T}

\newcommand{\wh}[1]{\widehat{#1}}

\newcommand{\ii}{\mc{I}}

\newcommand{\mm}{\mc{M}}
\newcommand{\maxid}{\frak{m}}
\newcommand{\ox}{\mc{O}_X}

\newcommand{\mc}[1]{\mathcal{#1}}
\newcommand{\on}[1]{\operatorname{#1}}
\newcommand{\ov}[1]{\overline{#1}}
\newcommand{\sS}{\on{sSch}}

\newcommand{\pist}{\pi_*} 

\newcommand{\lgr}{\overset{\sim}{\longrightarrow}}

\newcommand{\pp}{\mathbb{P}}
\newcommand{\z}{\mathbb{Z}}

\newcommand{\Ber}{\on{Ber}}
\newcommand{\Spec}{\on{Spec}}
\newcommand{\mfr}{\frak{M}_g}

\newcommand{\ff}{\mc{F}}

\newcommand{\tilx}{\tilde{X}}
\newcommand{\mcx}{\mc{X}}

\usepackage{setspace} % Package to control line spacing
\newtheorem{theorem}{Theorem}[section]
\newtheorem{lemma}[theorem]{Lemma}
\newtheorem{prop}[theorem]{Proposition}

\newtheorem{definition}[theorem]{Definition}

\newtheorem{remark}[theorem]{Remark}

\newcommand{\inv}{{}^{-1}}
\newcommand{\dd}{\mc{D}}
\newcommand{\tbold}{\mc{T}}

\newcommand{\Hom}{\on{Hom}}

\newcommand{\und}[1]{\underline{#1}}
\newcommand{\undhom}{\und{\Hom}}
\newcommand{\mct}{\mc{T}}
\newcommand{\Def}{\on{Def}}
\usepackage[
 backend=biber,
 style=numeric,
 sorting=ynt
 ]{biblatex}
\title{Algebraicity of supermoduli of curves via Artin’s criteria } 
\author{Nadia Ott\footnote{\emph{Email address}: \href{mailto:nmott@imada.sdu.dk}{nmott@imada.sdu.dk}\\ \indent Department of Mathematics, University of Southern Denmark, Odense, Denmark.
}}

\date{}

\begin{document}

\maketitle

\begin{abstract}
We apply the supergeometric analogue of Artin's algebraicity criteria \cite{ott2023artin} to prove algebraicity for four moduli problems in supergeometry: supercurves, super Riemann surfaces, stable supercurves, and stable super Riemann surfaces. 
The algebraicity of the moduli of (stable) super Riemann surfaces is known (see e.g.\ \cite{codogni2017moduli, felder2020moduli}), but we give a new proof by verifying the super Artin conditions. 
The algebraicity of the moduli of (stable) supercurves is new.
\end{abstract}

\tableofcontents

\section{Introduction}

Artin's theorems on approximation \cite{artin1969algebraic}, algebraization of formal deformations \cite{artin1969algebraization}, and algebraicity of stacks \cite{artin1974versal} provide general criteria for functors to be algebraic. In ordinary algebraic geometry this allows one to verify algebraicity abstractly, rather than by constructing explicit atlases or quotient presentations. By contrast, in supergeometry, until recently \cite{ott2023artin}, analogues of Artin’s theorems were not available, and constructions of moduli spaces were necessarily explicit.

In this paper we apply the super Artin criteria to several moduli problems in supergeometry: supercurves (Section \ref{section: moduli of supercurves}), super Riemann surfaces (Section \ref{section: moduli of super riemann surfaces is superstack}), stable supercurves (Section \ref{section: moduli of stable supercurves is an alg stack}), and stable super Riemann surfaces (Section \ref{section: moduli of stable super Riemann surfaces is a superstack}).

The moduli space of super Riemann surfaces $\mfr$, often called \emph{supermoduli space}, is a central object in supergeometry. Super Riemann surfaces arise in superstring theory as worldsheets of superstrings, and their moduli spaces appear as domains of integration for superstring amplitudes. The first mathematical construction of $\mfr$ as a super--orbifold was given in \cite{lebrun1988moduli}. Constructions for $g>1$ as algebraic superspaces were obtained in \cite{perez1997global}, and as a complex-analytic Deligne--Mumford superstack in \cite{codogni2017moduli}. 
Variants with Ramond punctures were shown to be algebraic superspaces in \cite{bruzzo2021supermoduli} for $g>1$, and Deligne--Mumford superstacks for $g=0$ and $r>1$ in \cite{ott2023supermoduli} via explicit constructions of \'etale atlases.
Recent work in \cite{bruzzo2025moduli} shows that the moduli of stable supermaps from super Riemann surfaces to projective superschemes form a Deligne-Mumford superstack.

In Section \ref{section: moduli of super riemann surfaces is superstack} we give a new proof of the algebraicity of $\mfr$ by verifying the Artin conditions.

\begin{theorem}
    The moduli space of genus $g \ge 2$ super Riemann surfaces is an algebraic superstack. 
\end{theorem}

Once algebraicity is established, the Deligne--Mumford property follows from the finiteness of automorphism groups (see e.g.\ \cite{deligneletter, faroogh2019existence}).

The space $\mfr$ is not compact. A compactification $\ovmfr$ by stable super Riemann surfaces was outlined by Deligne in a letter to Manin \cite{deligneletter}, up to verification of Artin’s conditions. Its existence as a Deligne--Mumford superstack was later established in \cite{felder2020moduli, faroogh2019existence} by constructing an \'etale atlas. In Section \ref{section: moduli of stable super Riemann surfaces is a superstack} we give a proof of algebraicity using the super Artin criteria.

\begin{theorem}
    The moduli space of stable super Riemann surfaces is an algebraic superstack. 
\end{theorem}

Again, the Deligne--Mumford property follows from the finiteness of automorphism groups \cite{deligneletter, faroogh2019existence}.

To verify the deformation-theoretic Artin conditions, we place the deformation theory of stable super Riemann surfaces into a superconformal $\mct^i$--sheaf framework (Theorem \ref{scf deformations of stable srs}) \footnote{A super version (not superconformal) of the $\mct^i$ sheaf formalism is given in Appendix \ref{appendix}.}. We introduce superconformal analogues of the classical $\mct^0$ and $\mct^1$ sheaves and show that the resulting deformation theory fits into exact sequences analogous to those governing deformations of schemes. The superconformal analogue of $\mct^0$ is the subsheaf $\mc{A}\subset\mct$ of superconformal vector fields, and the analogue of $\mct^1$ is a subsheaf $\mca\subset\mct^1$ defined in Section \ref{definition of aone}.

A super Riemann surface is a pair $(X/T,\mc{D})$, where $X/T$ is a smooth supercurve and $\mc{D}$ is a superconformal structure. Similarly, a stable super Riemann surface is a pair $(X/T,\delta)$ where $X/T$ is a stable supercurve and $\delta$ is a superconformal structure compatible with the singularities.

Here a \emph{supercurve} is a smooth proper superscheme of dimension $(1|1)$, and a \emph{stable supercurve} is a flat proper superscheme of dimension $(1|1)$ whose singularities are nodes (of Ramond or Neveu--Schwarz type). Thus the moduli of super Riemann surfaces and stable super Riemann surfaces may be viewed as moduli of supercurves and stable supercurves equipped with additional structure.

The moduli problems for supercurves and stable supercurves themselves have not been studied independently, although they form the natural base moduli spaces over which superconformal structures are imposed. In this paper we prove that both are algebraic superstacks.

\begin{theorem}
The moduli space of (stable) supercurves $\ovscg$ is an algebraic superstack.
\end{theorem}

The proof is given in Sections \ref{section: moduli of supercurves} and \ref{section: moduli of stable supercurves is an alg stack} by verifying the super Artin conditions. Although algebraicity for stable supercurves formally implies the result for smooth supercurves, we treat the smooth case separately by restricting in Section \ref{section: moduli of supercurves} to \emph{strongly projective supercurves}, i.e.\ supercurves admitting a polarization by a power of the canonical bundle.
This restriction simplifies the argument while omitting only a small class of supercurves (Lemma \ref{theorem: degree condition for strongly projective}). In particular, it does not exclude the supercurves underlying super Riemann surfaces (Lemma \ref{srs are strong projective}).

\paragraph{Future directions.}
For super Riemann surfaces and their stable compactification, the Deligne--Mumford property follows from algebraicity and the finiteness of automorphism groups. It is natural to ask whether the moduli of stable supercurves is Deligne--Mumford, and whether it admits a proper or compactification property analogous to the classical case.

More generally, the moduli of stable supercurves should play a foundational role for supergeometric moduli problems, serving as a base over which additional structures are imposed. It would be interesting to study moduli problems parametrizing supercurves together with additional geometric data (line bundles, vector bundles, coherent sheaves), and to determine whether these can be treated systematically using the super Artin criteria.

Another natural direction is to further develop the deformation theory of super Riemann surfaces in a more intrinsic manner. In this paper the sheaf $\mc{A}_{X/T}^1 \subset \mct_{X/T}^1$ is defined via explicit local computations at Ramond and NS nodes. It would be interesting to find a more intrinsic, functorial definition of $\mc{A}_{X/T}^1$ constructed directly from the quasi-superconformal structure $\delta$.

More broadly, many moduli problems in algebraic geometry and mathematical physics have not yet been developed in the supergeometric setting. We hope that this paper demonstrates that the super Artin criteria provide an effective tool for constructing such moduli spaces.

\section{Preliminaries}

Throughout the paper, $k$ will denote an algebraically closed field of characteristic zero. 

\subsection{Superalgebras and superschemes}

We use the heuristic term \emph{super object} in this section to mean the supergeometric analogue of a classical object (e.g.\ superalgebras, superschemes).
For any super object $A$, we write its $\z_2$-grading as
\[
A = A\plus \oplus A\minus.
\]
Morphisms of super objects are grading-preserving unless stated otherwise. If $X,Y$ are super objects, we write $\undhom(X,Y)$ for the set of all morphisms, both grading-preserving and grading-reversing. This set carries a natural $\z_2$-grading: the even component consists of grading-preserving morphisms, and the odd component consists of grading-reversing ones. We write $\on{Hom}(X,Y)$ for the even component.

We write \emph{superring} to mean a supercommutative Noetherian superring, and \emph{superalgebra} to mean a superring equipped with a $k$-superalgebra structure $k \to A$, which we refer to as its \emph{structure morphism}.
For a superring $A$, we will write $A_0$ for its \emph{full reduction}, i.e.\ the ordinary reduced ring obtained by quotienting $A$ by its nilradical. We write $J \subset A$ for the ideal generated by $A\minus$, and refer to it as the \emph{ideal of odd nilpotents}. The  \emph{bosonic reduction} of $A$ is the ordinary ring
\[
A_{\bos} := A/J.
\]

By a \emph{superscheme} we mean a locally superringed space $X=(|X|,\ox)$ with $\ox$ a sheaf of superalgebras and 
\[
X^{+} := (|X|,\ox^{+})
\]
an ordinary scheme over $k$.  We also require that the odd component $\ox\minus$ of $\ox$ is a quasi-coherent $\ox^{+}$-module.
For a superring $A$, the affine superscheme $\Spec(A)$ is defined in the usual way. Its bosonic reduction is $\Spec(A_{\bos})$. We denote by $\sS$ the category of Noetherian superschemes over $k$. In this paper, we assume all superschemes are locally Noetherian over the base.

A  superscheme over a base superscheme $T$ is a superscheme $X$ equipped with a morphism $X \to T$. All morphisms of superschemes $\pi:X \to Y$ are assumed to be \emph{local}: for every $x \in X$ with $y=\pi(x)$,
\[
\frak{m}_y \mc{O}_{X,x} \subseteq \frak{m}_x.
\]
The morphism $\pi$ is \emph{unramified} if the inclusion is an equality. 
We write $\jj \subset \ox$ for the ideal sheaf generated by $\ox\minus$. The \emph{bosonic reduction} of $X$ is the ordinary scheme
\[
\xbos := (|X|,\ox/\jj),
\]
and the quotient map $\ox \to \ox/\jj$ induces a closed immersion $\xbos \to X$.

A sheaf $\ff$ on a superscheme $X$ is an $\ox$-module if there exists for $\ox$-action on $\ff$: That is, for every $U \subset X$, an action of the superalgebra $\ox(U)$ on $\ff(U)$ compatible with restriction. A sheaf $\ff$ of $\ox$-modules is \emph{quasi-coherent} if for every open affine $U=\Spec(A) \subset X$ there exists an $A$-module $M$ such that $\ff \vert_U$ is the sheaf $M^{\sim}$ associated to $M$. It is \emph{coherent} if for every $U$, the $A$-module $M$ is finitely-generated, and \emph{locally free of rank $(m|n)$} if in addition $M$ is isomorphic to the free $A$ module of rank $(m|n)$: 
\[ A^{m|n}:= A^{\oplus n} \oplus \Pi A^{\oplus m}. \]
The sheaf associated to $A^{m|n}$ on $U$ is $(\ox \vert_U)^{m|n} $, so this means
\[ \ff \vert_U \cong (\ox \vert_U)^{m|n}.\]

A morphism $\pi:X \to Y$ of superschemes is \emph{flat} if for every point $x$ in $X$ and $y=\pi(x)$ in $Y$ the induced morphism  on local rings \[ \mc{O}_{Y,y} \to \mc{O}_{X,x} \]  makes $\mc{O}_{X,x}$  a flat $\oy$-module. Equivalently, the pullback functor
\[
\pipull(\ff) := \pi^{-1}\ff \otimes_{\pi^{-1}\oy} \ox 
\]
is exact for all sheaves $\ff$ off $\ox$-modules. 
The morphism $\pi:X \to Y$ is \emph{smooth of relative dimension $(m|n)$} if it is locally of finite presentation, flat, and the sheaf of relative differentials $\omegaxt$ is locally free of rank $(m|n)$. It is \emph{\'etale} if it is flat, of finite type, and unramified.
A morphism $f:X \to T$ is \emph{projective} (or \emph{superprojective}) if there exists a coherent sheaf $\mm$ on $T$ such that $f$ factors as
\[
X \xhookrightarrow{i} \pp(\mm) \xrightarrow{\pi} T,
\]
where $i$ is a closed immersion over $T$ and
\[
\pp(\mm) := \underline{\on{Proj}}_T(\mc{S}ym_{\ot}^{\bullet}\mm).
\]

We say that $X/T$ is \emph{strongly projective} if $\mm$ is locally free of finite rank. 

A line bundle $\mcl$ on $X/T$ is \emph{relatively very ample} if there exists a coherent sheaf $\mm$ on $T$ and a closed immersion
\[
i:X \hookrightarrow \pp(\mm)
\]
over $T$ such that
\[
\mcl \cong i^*\mc{O}_{\pp(\mm)}(1).
\]
It is \emph{relatively ample} if some power $\mcl^{\otimes n}$ is relatively very ample.
A line bundle $\mcl$ on $X/T$ is \emph{strongly relatively very ample} if the following conditions hold: 
\begin{enumerate}
\item[(i)] $f_*\mcl$ is locally free of finite rank.
\item[(ii)] The canonical map $f^*f_*\mcl \to \mcl$ is surjective. 
\item[(iii)] The induced morphism $i_{\mcl}:X \to \pp(f_*\mcl)$ is a closed immersion over $T$.
\end{enumerate}
A line bundle $\mcl$ is \emph{strongly relatively ample} if some tensor power of it is strongly relatively very ample.

The property of being strongly relatively very ample can be checked after restriction to the bosonic reduction: 

    \begin{prop}[Proposition A.2, \cite{felder2020moduli}] \label{propatwo}
        Let $X \to T$ be a morphism of superschemes. A line bundle $\mcl$ on $X$ is strongly relatively very ample over $T$ if 
        $\mcl|_{\xbos}$ is strongly relatively very ample on $\xbos$ over $\tbos$.
        \end{prop}

\subsection{Superstacks}
Let $\sS$ denote the category of Noetherian superschemes over $k$. 
\smallskip

The \emph{functor of points} of a superscheme $X$ is the contravariant functor
\[
h_X : \sS \longrightarrow \on{Set}
\]
sending a superscheme $T$ to the set \[ h_X(T):=\on{Hom}(T,X) \] of grading-preserving morphisms from $T$ to $X$.
A functor $F : \sS \to \on{Set}$ is \emph{representable} if there exists a superscheme $X$ such that $F \cong h_X$.

The functor of points formalism provides a precise formulation of a moduli problem. Given a class of objects together with a notion of isomorphism, define a functor
\[
M \colon \sS \longrightarrow \on{Set}
\]
sending a superscheme $T$ to the set of isomorphism classes of families parameterized by $T$. If $M$ is representable, then the moduli problem is represented by a superscheme, and we call it a \emph{fine moduli space}. 

Fine moduli spaces tend to not exist when the moduli objects have non-trivial automorphisms. This is especially common in supergeometry. 
The standard remedy is to incorporate isomorphisms into the moduli problem itself. Instead of recording only isomorphism classes, one assigns to each $T$ the groupoid of families over $T$, keeping track of both objects and their isomorphisms. A category in which every morphism is an isomorphism is called a groupoid.
One may encode such a moduli problem by a pseudo--functor
\[
\mcx \colon \sS \longrightarrow \on{Groupoids}
\]
where the category of superschemes $\sS$ is equipped with the big \'etale topology. 

Every such pseudo--functor gives rise to a \emph{category fibered in groupoids}. A category fibered in groupoids over $\sS$ is a functor
\[
p : \mcx \longrightarrow \sS
\]
satisfying the following conditions:

\begin{enumerate}
\item[(i)] For every morphism $f : S \to T$ in $\sS$ and every object $X$ of $\mcx$ over $T$, there exists a cartesian morphism $\tilde f : Y \to X$ lying over $f$.
\item[(ii)] Cartesian morphisms are stable under composition.
\end{enumerate}

A category fibered in groupoids $\mcx \to \sS$ is called an \emph{\'etale superstack} if the following descent conditions are satisfied:

\begin{itemize}
\item[(i)] (\emph{Descent for morphisms}.) For every $U \in \sS$ and $X,Y \in \mcx(U)$, the presheaf
\begin{align*}
\mathbf{\on{Isom}}(X,Y) : \sS/U &\longrightarrow \on{Set} \\
V &\longmapsto \{\alpha : X|_V \lgr Y|_V \text{ is an isomorphism in } \mcx(V)\}
\end{align*}
is a sheaf on the \'etale site $\sS/U$ of superschemes over $U$. 

\item[(ii)] (\emph{Descent for objects}.) For every \'etale cover $\{U_i \to U\}$, objects $X_i \in \mcx(U_i)$, and isomorphisms
\[
\alpha_{ij} : \prijpull X_i \lgr \prjipull X_j
\]
on overlaps $U_{ij}=U_i \times_U U_j$ satisfying the cocycle condition
$
\alpha_{ik} = \alpha_{jk} \circ \alpha_{ij}
$
on triple overlaps $U_{ijk}$, there exists an object $X \in \mcx(U)$ together with isomorphisms
$
\alpha_i : \pripull X \lgr X_i
$
whose restrictions to overlaps satify
$\alpha_{ij} = \prjipull \alpha_j \circ (\prijpull \alpha_i)^{-1}.$
\end{itemize}

A superstack $\mcx$ is \emph{algebraic}, or \emph{Artin}, if the diagonal morphism
\[
\Delta:\mcx\longrightarrow\mcx\times\mcx
\]
is representable, and there exists a smooth surjective morphism $\phi:U\to\mcx$ from a superscheme $U$.
An algebraic superstack is \emph{Deligne-Mumford} if $\phi$ is \'etale.

We now state the supergeometric analogue of  Artin's algebraicity theorem from \cite{ott2023artin}. The notation of the theorem will be used throughout the paper.
 
\begin{theorem}
    \label{theorem: introduction super artin statement} Let $\mcx$ be an \'etale superstack of the category $\sS$ of Noetherian superschemes over $k$. Then $\mcx$ is an algebraic superstack locally of finite type over $k$ if and only if the following conditions hold:
    
    \begin{enumerate}
    \item[\emph{(A1)}.] \emph{Representability of the diagonal}: The diagonal morphism $\Delta: \mcx \to \mcx \times \mcx$ is represented by an algebraic superspace locally of finite type. Equivalently, for every superscheme $T$ and $X,X' \in \mcx(T)$, the sheaf \[ \underline{\on{Isom}}(X,X'): \sS/T \to \on{Set} \] sending 
    $S \to T$ to the set of isomorphisms $X_S \lgr X_S'$ is
    representable by an algebraic superspace locally of finite type over $T$.
    \item[\emph{(A2)}.] \emph{Limit preserving}: For every directed system of superalgebras $(A_i \to A_j)$ with colimit $A= \varinjlim A_i$ the natural map
    \begin{equation} \label{intro: colimit preserving} \varinjlim \mcx(A_i) \to \mcx(A)\end{equation}
    is an equivalence of categories.
    \item[\emph{(A3)}.] \emph{Generalized Schlessinger conditions}: Let $A' \to A$ be a square-zero extension of superalgebras with kernel $M=\ker(A' \to A)$ a finitely-generated module over the full reduction $A_0:=A/\frak{N}$ of $A$. Let $X \in \mcx(A)$, and let $X_0$ denote the pullback of $X$ by $A \to A_0$. Define
    \begin{equation}
    \label{intro: dxm definition}
    D_X(M):= \ov{\mcx}_X(A[M]).
    \end{equation}
    \begin{enumerate}
        \item[\emph{(S1)a}.] For every surjection $B \to A$, the natural map
        \begin{equation}
        \label{intro: schless sonea}
        \ov{\mcx}_X(A' \times_A B) \to \ov{\mcx}_X(A') \times \ov{\mcx}_X(B)
        \end{equation}
        is a surjection. 
        \item[\emph{(S1)b}.] Let $X' \in \mcx(A')$ and let $X$ be its pullback by $A' \to A$. Then the canonical map
        \begin{equation}
        \label{intro: schless soneb}
        D_{X'}(M) \to D_X(M)
        \end{equation}
        is a bijection.
        \item[\emph{(S2)}.] $D_{X_0}(M)$ is a finitely generated $A_0$-module. 
    \end{enumerate}
    \item[\emph{(A4)}.] \emph{Effectivity}: Let $(R, \frak{m})$ be complete, local, Noetherian superalgebra. The natural map 
    \begin{equation} \label{intro: effectivity}  \mcx(R) \to \varprojlim \mcx( R/\frak{m}^{n+1})\end{equation}
    is an equivalence of categories.   
    \item[\emph{(A5)}.] \emph{Coherent obstruction theory}: The obstruction module $\mc{O}_X(M)$ is a coherent $A$-module. 
    \item[\emph{(A6)}.] \emph{Constructibility}:  Let $T=\Spec(A)$, and let $X \in \mcx(T)$. For every integral subscheme $S=\Spec(B) \subset T$ there exists a dense Zariski open $V=\Spec(C) \subset S$ such that for every closed point $t \in V$ the natural maps
    \begin{equation} \label{construct d}
    D_X(B^{1|1}) \otimes k(t) \;\longrightarrow\; D_X(k(t)^{1|1}),
    \end{equation}
    \begin{equation} \label{construct ob}
    O_X(B^{1|1}) \otimes k(t) \;\longrightarrow\; O_X(k(t)^{1|1})
    \end{equation}
    are bijective and injective, respectively.
    
    \end{enumerate}

    \item[\emph{(A7)}.] \emph{Compatible with completion}: 
    Let $\wh{A_0}$ denote the completion of $A_0$ along a maximal ideal $\maxid \subset A_0$. The natural map
    \begin{equation} \label{intro: compatible with completion}
    D_{X_0}(M) \otimes_{A_0} \wh{A_0} \;\longrightarrow\; \varprojlim_n D_{X_0}(M/\maxid^{n+1} M) 
    \end{equation}
    is a bijection. 
    
    \item[\emph{(A8)}.] \emph{Compatible with \'etale localization}:  Let $e: S=\Spec(B) \to T=\Spec(A)$ be an {\'e}tale morphism, and let $S_0$ and $T_0$ denote their respective full reductions. Let $X\in\mcx(T)$, and let $Y \in\mcx(S)$ be the base change of $X$ along $e$.  The natural maps
    \begin{equation} \label{intro: d is comp with etale morph}
    D_{Y_0}(M \otimes B_0) \to D_{X_0}(M)\otimes B_0 .
    \end{equation}
    \begin{equation} \label{intro: o is comp with etale morph}
        O_{Y_0}(M \otimes B_0) \to O_{X_0}(M)\otimes B_0 .
        \end{equation}
    are bijections. 
    
    \end{theorem}

\section{Supercurves} \label{section: supercurves}

A \emph{supercurve} $X$ is a smooth, proper, and connected superscheme over $k$ of dimension $(1|1)$ such that its bosonic reduction $C:= \xbos=(|X|, \ox/\jj)$ is an ordinary smooth, connected curve over $k$.
    
 The \emph{genus} $g$ of a supercurve $X$ is the genus $g$ of its underlying bosonic curve $C$. From now on we assume $g \ge 2$. 

\begin{definition} 
    A \emph{family of genus $g$ supercurves over a superscheme $T$} is a smooth, proper morphism of superschemes $\pi: X \to T$ of relative dimension $(1|1)$ such that each geometric fiber $X_t$ is a genus $g$ supercurve over $k(t)$ and such that the bosonic reduction 
    \[ \pibos: \xbos\to \tbos\]
    is a family of ordinary smooth curves of genus $g$ over $\tbos$.
    An \emph{isomorphism of families over the same base $T$} from $X$ to $X'$ is an isomorphism of superschemes $\phi: X \to X'$ commuting with the projections to $T$.  
    \end{definition}  

\begin{lemma} \label{lemma: supercurves correspondence with picard}
    There is a one to one correspondence between supercurves $X/T$ over purely bosonic schemes $T$ and pairs $(C/T, L)$ where $C/T$ is an ordinary family of curves and $L$ is a line bundle on $C$. 
    \end{lemma}

    \begin{proof} When $T$ is an ordinary scheme, $\jj \subset \ox$ satisfies $\jj^2=0$. This has various implications: (1) the $\ox$-module structure on $\jj$ factors through $\oc=\ox/\jj$, \[ \ox= \oc \oplus \jj. \] 
and $\jj$ is the conormal sheaf of the embedding $C \hra X$. Since both $C$ and $X$ are smooth of dimension $(1|0)$ and $(1|1)$, respectively, the conormal sheaf $\jj$ is a locally free sheaf of $\oc$-modules of rank $(0|1)$. 
Thus, $X$ determines a pair \[ (C/T,L):=(\xbos/T, \Pi \jj). \]

Conversely, given a pair $(C/T, L)$, define the superscheme
     \[ X=(|C|, \ox \oplus \Pi L). \]
     This is clearly a supercurve with $\jj= \Pi L$. 
   \end{proof}

    \begin{definition} We say that a family of supercurves $X \to T$ is \emph{strongly projective} if there exists a non-zero integer $n$ such that \[ \omega_{X/T} ^{\otimes n}:= \Ber(\omegaxt) ^{\otimes n} \] is relatively very ample on $X/T$.  
    \end{definition}
       Clearly a supercurve $X/T$ is strongly projective if and only if either $\duaxt$ or $\duaxt \inv$ is relatively ample on $X/T$. 
\smallskip

We will now use Proposition \ref{propatwo} to determine which supercurves are strongly projective by computing the degree of the restriction $\duaxt \vert_C$ to the bosonic reduction $C$ of $X$. 
Since the criterion in \eqref{propatwo} is stated after restriction to the bosonic reduction, we may assume $T$ is a purely bosonic scheme. Then $\jj$ is a locally free sheaf on $C$ of rank $(0|1)$ and equal to the conormal bundle to the embedding $C \hra X$. Applying $\Ber$ to the conormal sequence of $C \hra X$, we find
        $\duaxt\vert_C  \cong \omega_{C/T}\otimes \jj\inv$ which implies: 
        \begin{equation}\label{degree formula}
        \deg(\duaxt\vert_C)=(2g-2)-\deg(\jj).
        \end{equation}

 % In use the next Proposition \ref{propatwo} together  with the fact from ordinary algebraic geometry that a line bundle $L$ on a curve $C$ is strongly very ample if the relative degree of $L$ is at least $2g+1$
 %    to determine which supercurves are strongly projective by considering the degree of the restriction of the dualizing bundle to the underlying bosonic curve.  
 %    \begin{lemma}[Proposition A.2] \label{propatwo}
 %        Let $X \to T$ be a morphism of superschemes. A line bundle $\mcl$ on $X$ is strongly relatively very ample over $T$ if 
 %        $\mcl|_{\xbos}$ is strongly relatively very ample on $\xbos$ over $\tbos$.
 %        \end{lemma}
 % Combining this with the fact from ordinary algebraic geometry that a line bundle $L$ on a curve $C$ is strongly very ample if the relative degree of $L$ is at least $2g+1$, we find: 

    \begin{lemma}
    \label{theorem: degree condition for strongly projective} A supercurve $X/T$ is strongly projective if and only if 
    \[ \deg(\duaxt \vert_C)\neq 0. \]
   where $C$ denotes the bosonic reduction of $X$. 
    \end{lemma}

    \begin{proof}

    Suppose $X/T$ is strongly projective. Fix a nonzero integer $n$ such that $\mcl=\duaxt^{\otimes n}$ is relatively very ample.  Set $\mce:=\pist(\duaxt^{\otimes n})$ and let $i$ denote the induced closed immersion
        \[
        i\colon X\hra \pp_T(\mce),
        \qquad
        \duaxt^{\otimes n} \cong \ipull \mc{O}_{\pp_T(\mce)}(1)
        \]
        over $T$. 
   Denote by $\phi: C \hra X \to \pp(\mce)$ the composition with the canonical embedding $C\hra X$. Since $ \duaxt^{\otimes n}\vert_C \cong \phipull \mc{O}_{\pp_T(\mce)}(1)$,  $\duaxt^{\otimes n}\vert_C$ is relatively very ample on $C/T$, and hence
        \[
        \deg(\duaxt^{\otimes n} \vert_C)=n\cdot \deg(\duaxt\vert_C) > 0,
        \]
   from which it follows that $\deg(\duaxt\vert_C)\neq 0$.
   \smallskip
   
     Conversely,  assume $\deg(\duaxt\vert_C)\neq 0$, and choose a nonzero integer $n$ such that \[ \deg(\duaxt\vert_C^{\otimes n})\ge 2g+1. \]  
     
     Then $\duaxt^{\otimes n}\vert_C$ is  strongly relatively very ample on $C/T$ Lemma \ref{degree condition for sva}, and by Proposition \ref{propatwo},  $\duaxt^{\otimes n}$ is strongly relatively very ample on $X/T$. 
        \end{proof}

\paragraph{Projective system of supercurves} \label{proj notation}

We will now fix notation used in the proofs of effectivity (A4) throughout the paper. At the end of a section we state a lemma used in the proof of effectivity for the moduli of strongly projective supercurves. The arguments of the lemma  are used in later sections.

Let $(R,\maxid)$ be a complete local Noetherian $k$-superalgebra. Set
\[
R_j = R/\maxid^{j+1}, \qquad T = \Spec(R), \qquad T_j = \Spec(R_j).
\]
The natural surjections
\[
R \twoheadrightarrow R_j \twoheadrightarrow R_i
\]
induce morphisms
\[
T_i \to T_j \to T.
\]

Let $(X_j,\, X_i \to X_j)$ be a projective system of strongly projective supercurves over the system $(T_j,\, T_i \to T_j)$. For simplicity we denote all structure morphisms $X_j \to T_j$ by $\pi$. The curves are related by base change
\[
X_i \cong X_j \times_{T_j} T_i,
\]
and the transition morphisms
\[
\prij : X_i \to X_j
\]
are the projections from this fiber product. In particular each $\prij$ is a nilpotent thickening and all elements of such a projective system have the same topological space.

Let $\mcl$ be any coherent sheaf on a superscheme whose formation commutes with base change. 
On a projective system $(X_j, X_i \to X_j)$ these sheaves form a projective system $\mclbull=(\mcl_j,\, \mcl_j \to \mcl_i)$ as follows. For each $j \ge i$, base change gives a canonical isomorphism
\[
\prjipull \mcl_j \xrightarrow{\sim} \mcl_i.
\]
We define the transition morphism $\mcl_j \to \mcl_i$ to be the adjoint of this isomorphism, i.e. the composite
\begin{equation} \label{standard trans}
\mcl_j
\longrightarrow
\pr_{ij*}\prjipull \mcl_j
\longrightarrow
\pr_{ij*}\mcl_i.
\end{equation}
Since $\prij$ is a nilpotent thickening, all $X_j$ have the same underlying topological space, and we will identify sheaves on $X_i$ with their direct images on $X_j$. This gives us morphisms $\mcl_j \to \mcl_i$ that we refer to as \emph{standard transition maps}.

If $\mclbull = (\mcl_j, \mcl_j \to \mcl_i)$ is a system of canonically defined coherent sheaves on $\xbull = (X_j, X_i \to X_j)$, we say that $\mclbull$ has property $P$ if each $\mcl_j$ has property $P$. If $F$ is a functor on coherent sheaves, 
we write $F(\mclbull)$ for the induced system
\[
F(\mclbull) = (F(\mcl_j),\, F(\mcl_j) \to F(\mcl_i)).
\]
\begin{lemma}
\label{lemma: serre vanishing}
Using the above notation, for any projective system of strongly projective supercurves $(X_j, X_i \to X_j)$, there exists an integer $N$ such that the system
\[
\mclbull=(\mcl_j, \mcl_j \to \mcl_i), \qquad \mcl_j:= \omega_{X_j/T_j}^{\otimes N}
\]
is strongly relatively very ample and satisfies $R^1 \pi_{*}\mcl_{\bullet} = 0$.
\end{lemma}

\begin{proof}

Choose an integer $n$ such that $\omega_{X_0/T_0}^{\otimes n}$ is strongly relatively very ample. Since this property can be checked after restricting along the nilpotent extension $X_0 \to X_j$, it follows that $\duaxjtj^{\otimes n}$ is strongly relatively very ample for all $j$, and
\[
\mclbull=(\mcl_j, \mcl_j \to \mcl_i), \qquad \mcl_j:= \duaxjtj^{\otimes n}
\]
is a strongly relatively very ample projective system
of line bundles on $\xbull=(X_j, X_i \to X_j)$.

We now replace $n$ by a positive multiple to ensure vanishing of higher cohomology. Since $X_0/T_0$ is projective with polarization $\mcl_0$, the super version of Serre vanishing \cite{bruzzo2021supermoduli} implies that there exists an integer $m \ge 1$ such that
\[
R^1 \pist \omega_{X_0/T_0}^{\otimes mn} = 0.
\]
Redefine
\[
\mclbull=(\mcl_j, \mcl_j \to \mcl_i), \qquad \mcl_j:= \duaxjtj^{\otimes mn}.
\]

We claim that this system is strongly relatively very ample and satisfies $R^1 \pist \mclbull = 0$. The first statement holds since taking a positive tensor power preserves strong relative very ampleness.

For the vanishing of $R^1 \pist \mclbull$, we argue by induction on $j$. The base case $R^1 \pist \mcl_0=0$ holds by construction. Assume $R^1 \pist \mcl_{j-1}=0$, and set $M=\maxid^j/\maxid^{j+1}$. Since $X_j$ is $T_j$-flat, we have a short exact sequence
\begin{equation} \label{first sequence}
0 \to M \otimes_{R_j} \oxj \to \oxj \to \mco_{X_{j-1}} \to 0.
\end{equation}
after pulling back $0 \to M \to R_j \to R_{j-1} \to 0$. 

Because $\maxid \cdot M = 0$, the $R_j$-module structure on $M$ factors through $k=R/\maxid$, so $M$ is a super vector space of some finite dimension $(p|q)$. Thus,
\[
M \otimes_{R_j} \oxj
\cong
M \otimes_k (R_0 \otimes_{R_j} \oxj)
\cong
M \otimes_k \oxnot
\cong
\oxnot^{p|q},
\]
since $R_0 \otimes_{R_j} \oxj \cong \oxnot$. Then tensoring \eqref{first sequence} with $\mcl_j$ gives
\[
0 \to \mcl_0^{p|q} \to \mcl_j \to \mcl_{j-1} \to 0
\]
where we used that
$\mcl_{j-1} \cong  \mcl_j \otimes_{\oxj} \mc{O}_{X_{j-1}}$ since any power of the dualizing sheaf commutes with base change. 
Applying $\pist$, we obtain the exact sequence
\[
0 \to \pist \mcl_0^{p|q} \to \pist \mcl_j \to \pist \mcl_{j-1}
\to R^1 \pist \mcl_0^{p|q} \to R^1 \pist \mcl_j \to R^1 \pist \mcl_{j-1} \to \cdots .
\]

Since $R^1 \pist \mcl_0^{p|q}=0$ and $R^1 \pist \mcl_{j-1}=0$ by the inductive hypothesis, it follows that $R^1 \pist \mcl_j=0$.
\end{proof}

\subsection{Deformation theory of strongly projective supercurves}

Let $A' \to A$ be a square-zero extension of superalgebras with kernel $M=\ker(A' \to A)$ a finitely generated  $A$--module. Set $T=\Spec(A)$, and $T'=\Spec(A')$. 
\smallskip

A \emph{deformation of $X/T$ over $T'$} is a flat, proper superscheme $X'/T'$ such that
\[
X' \times_{T'} T \cong X.
\]
A \emph{projective deformation} is a deformation that is a strongly projective supercurve. An isomorphism of (projective) deformations is a $T'$-linear isomorphism restricting to the identity on the central fiber. 

\begin{lemma} \label{lemma:defs of supercurves}
Every deformation of a stable supercurve $X/T$ is projective. In particular, the set of isomorphism classes of projective deformations of $X/T$ is a torsor for the group
\[
H^1(X, \mct_{X/T} \otimes_{A} M)^{+},
\]
where $\mct_{X/T}$ is the relative tangent sheaf.
\end{lemma}

\begin{proof} We need to show that every
deformation $X'/T'$ of $X/T$ is a strongly projective supercurve. It is flat by definition, so smoothness can be checked on geometric fibers. Since $T \hra T'$ is a nilpotent extension, the only geometric fiber is the central fiber $X/T$, which is smooth by assumption. 
$X'/T'$ is strongly projective since $\omega_{X'/T'}$ pulls back to the relatively ample line bundle $\omega_{X/T}$ along the nilpotent extension $X \hra X'$ and the property of being relatively ample is preserved under nilpotent thickening.

For the second claim, since $X/T$ is smooth, $\mct_{X/T}^1 = \mct_{X/T}^2=0$ and so there are no obstructions to deforming $X/T$ along square-zero extensions. Then, by Theorem \ref{from Hart a}, the deformations of  $X/T$ over $T'$ form a torsor for the group
\[
H^1(X, \mct_{X/T} \otimes_{A} M)^{+}.
\]
\end{proof}

\subsection{Moduli of strongly projective supercurves is an algebraic superstack} \label{section: moduli of supercurves}

Let $\scg$ be the category over $\sS$ whose fiber over a superscheme $T$ is the groupoid $\scg(T)$ of genus $g$ supercurves $X \to T$, with morphisms given by isomorphisms. 
Pullback in $\scg$ is defined by base change: for a morphism of superschemes $f:S \to T$ and a supercurve $X \to T$, the pullback is
\[
X_S := X \times_T S \to S,
\]
which is again a family of supercurves of genus $g$.

Let $\cc$ denote the category over $\sS$ whose fiber $\cc(T) \subset \scg(T)$ consists of strongly projective supercurves over $T$, with the same morphisms as in $\scg(T)$. $\cc$ is clearly a full subcategory of $\scg$. It is closed under pullback: if $X \to T$ is strongly projective, so that $\duaxt^{\otimes n}$ is strongly relatively very ample for some $n \ge 1$, then for any morphism $f:S \to T$ we have
\[
\duaxs^{\otimes n} \cong p^*(\duaxt^{\otimes n}), \qquad p: X_S \to X
\]
and hence $\duaxs^{\otimes n}$ is again strongly relatively very ample. Thus $X_S \to S$ is strongly projective.

\begin{prop}\label{proposition: supercurves is a superstack}
     $\cc$ is an \'etale superstack.
\end{prop}

\begin{proof}

\noindent\emph{(i) Descent for isomorphisms.}
Let $X,Y \in \cc(U)$ and let $\{U_i \to U\}$ be an \'etale cover. Suppose we are given isomorphisms
\[
\alpha_i : \pripull X \;\lgr\; \pripull Y, 
\]
compatible on overlaps. Since morphisms of superschemes satisfy fppf, and thus \'etale descent \cite[Proposition A.28]{bruzzo2021supermoduli}) the $\alpha_i$ glue to a unique morphism $\alpha : X \to Y$ over $U$. It remains to check that $\alpha=(|\alpha|,\alpha^\#)$ is an isomorphism. The map $|\alpha|:|X|\to |Y|$ is a homeomorphism because each $|\alpha_i|$ is a homeomorphism and these agree on overlaps. The morphism $\alpha^\#:\mc{O}_Y\to \alpha_*\mc{O}_X$ is an isomorphism because its kernel and cokernel vanish after pulling back along the \'etale cover $\pripull Y\to Y$, and thus vanish globally by \'etale descent for quasi-coherent sheaves \cite[Proposition A.29]{bruzzo2021supermoduli}.
\smallskip

\noindent\emph{(ii) Descent for objects.}
Let $\{U_i \to U\}$ be an \'etale cover, and suppose for each $i$ we are given a supercurve $X_i/U_i$ together with isomorphisms
\[
\alpha_{ij} : \prijpull X_i \;\lgr\; \prjipull X_j
\]
satisfying the cocycle condition. 

There exists a non-zero integer $n$ such that $\mcl_i:=\omega_{X_i/U_i}^{\otimes n}$ is strongly relatively very ample  for all $i$. Set
\[
\mce_i = \pist \mcl_i .
\]

Since $\prij:U_i\times_U U_j\to U_i$ is \'etale, flat base change gives a canonical isomorphism
\[
\widetilde{\sigma}_{ij}:\pr_{ij}^*\mce_i \lgr \pr_{ji}^*\mce_j
\]
satisfying the cocycle condition. By \'etale descent for coherent sheaves there exists a unique coherent sheaf $\mce$ on $U$ such that $\pripull\mce\cong \mce_i$. The canonical maps $\pp(\mce_i) \to \pp(\mce)$ give an \'etale cover

\[
 \{ \pp(\mce_i) \longrightarrow \pp(\mce) \}_{i \in I}.
\]
The canonical immersions $X_i \to \pp(\mce_i)$ agree on overlaps by construction, and thus by \'etale descent for closed immersions \cite[Proposition A.30]{bruzzo2021supermoduli} there exists a closed superscheme $X$ over $U$ such that $\pripull X\cong X_i$ as closed subschemes of $\pp(\mce_i)$. 

The superscheme $X/U$ is clearly a supercurve since smoothness, properness, and relative dimension $(1|1)$ are \'etale local on the base. Furthermore, the bosonic reduction $C=\xbos$ of $X$ is a smooth ordinary curve over $U_{\bos}$ because $\{C_i\to C\}$ is an \'etale cover of $C$ and each $C_i=(X_i)_{\bos}$ is smooth over $U_i$ by assumption.
Finally, $\omega_{X/U}^{\otimes n}$ is relatively very ample on $X$. Indeed, let $\tau: X \hra \pp(\mce)$ denote the canonical immersion, and set
\[
\mcl := \tau^*\mc{O}_{\pp(\mce)}(1).
\]
Then $\pripull\mcl$ and $\pripull\duaxt^{\otimes n}$ are canonically isomorphic since both are canonically isomorphic to $\duaxiti^{\otimes n}$. These identifications agree on overlaps and therefore glue to a global isomorphism
\[
\mcl \cong \duaxt^{\otimes n}
\]
on $X$.
\end{proof}

\begin{prop}
\label{proposition: rep of diagonal for supercurves}
The superstack $\cc$ satisfies condition (A1) of Theorem \ref{theorem: introduction super artin statement}.
\end{prop}

\begin{proof}
Since $X \to T$ and $X' \to T$ are projective morphisms, their fiber product
\[
X \times_T X' \to T
\]
is also projective. In particular, $X \times_T X'$ is superprojective over $T$ and admits a relatively very ample line bundle. Therefore, by  \cite[Theorem 4.3]{bruzzo2023notes}, the super Hilbert functor of closed subschemes of $X \times_T X'$ that are flat and proper over $T$ is representable by a $T$-superscheme. We set
\[
\hh := \Hilb(X \times_T X'/T).
\]

Let $S$ be a $T$-superscheme. Any morphism $\phi : X_S \to X_S'$ determines a closed immersion
\begin{equation}\label{graph map}
X_S \xrightarrow{(\on{id},\phi)} X_S \times_S X'_S = (X \times_T X')_S .
\end{equation}
Let $\Gamma_S(\phi)$ denote the corresponding closed subscheme of $(X \times_T X')_S$. Since $\Gamma_S(\phi) \cong X_S$ flat and proper over $S$, it defines an $S$-point of $\hh$. This gives an injective map
\[
\Gamma_S : \uniso(X,X')(S) \hookrightarrow \hh(S)
\]
which is functorial in $S$, and therefore a natural transformation
\[
\Gamma : \uniso(X,X') \to \hh,
\]
making $\uniso(X,X')$ a subfunctor of $\hh$.

We claim that $\uniso(X,X')$ is representable by an open subscheme of $\hh$. Let $(S \to T,\xi)$ be an object with $\xi \in \hh(S)$. Then there exists an open subscheme $U_\xi \subset S$ such that a morphism $f:S' \to S$ factors through $U_\xi$ if and only if $f^*\xi \in \uniso(X,X')(S')$.

Let
\[
Z \subset (X \times_T X') \times_T \hh
\]
denote the universal closed subscheme. Thus for any $T$-superscheme $S$ and any $\xi \in \hh(S)$, the corresponding subscheme $Z_S \subset (X \times_T X')_S$ is obtained by pulling back $Z$ along the morphism $S \to \hh$ corresponding to $\xi$.

For $\xi \in \hh(S)$, the subscheme $Z_S$ is the graph of an isomorphism $X_S \to X_S'$ if and only if the projections
\[
p : Z_S \to X_S, 
\qquad 
q : Z_S \to X_S'
\]
are isomorphisms over $S$. Indeed, if $p$ is an isomorphism then the composite
\[
X_S \cong Z_S \xrightarrow{q} X_S'
\]
defines a morphism whose graph is $Z_S$, and the additional condition that $q$ is an isomorphism ensures that this morphism is an isomorphism.
Since $p$ and $q$ are projective morphisms, the locus where they are isomorphisms is open, and the result follows by the argument of \cite[Theorem 5.22.(b)]{fantechi2005fundamental}. \end{proof}

\begin{prop} \label{atwo: colimit preserving}
The superstack $\cc$ satisfies condition (A2) of Theorem \ref{theorem: introduction super artin statement}.
\end{prop}

\begin{proof}
Let $X/A$ be an object of $\cc(A)$. Since $X$ is finite type over the Noetherian superalgebra $A$, it is of finite presentation over $A$, and thus admits a finite affine open cover $\{U_{\alpha}=\Spec(B_{\alpha}) \}_{\alpha \in \mc{A}}$ with each $B_{\alpha}$ a finitely generated $A$-superalgebra. Choose a presentation
\[
B_{\alpha} \cong A[x]/I
\]
where $x=(x_1,\dots,x_n)$ and the $x_i$ can be even or odd. Since $A$ is Noetherian, $A[x]$ is Noetherian, and the ideal $I$ is finitely generated, $I=(f_1,\dots,f_r)$.
The coefficients of each polynomial $f_k \in I$ lie in $A=\varinjlim A_i$, and since there are only finitely many polynomials, there exists $i$ such that $f_k \in A_i$ for all $1 \le k \le r$. Thus,
\[
B_{\alpha} \cong \bigl(A_i[x]/I\bigr)\otimes_{A_i} A.
\]
Since the cover is finite, there exists some maximal index $i$ such that 
$B_{\alpha} \cong A_i[x]/I \otimes_{A_i} A $ for all $\alpha \in \mc{A}$.

The gluing morphisms on overlaps are $A$-superalgebra maps
\[
\phi_{\alpha\beta}^\# : B_\beta \longrightarrow B_\alpha .
\]
Since $B_\alpha \cong B_{\alpha,i}\otimes_{A_i} A$ and
$B_\beta \cong B_{\beta,i}\otimes_{A_i}A$, and
$B_{\alpha,i},B_{\beta,i}$ are finitely presented $A_i$-superalgebras, we have be \cite[Lemma 7.13]{faroogh2019existence}
\[
\Hom_A(B_\beta,B_\alpha)
\;\cong\;
\varinjlim_{j\ge i}\Hom_{A_j}(B_{\beta,j},B_{\alpha,j}).
\]
Therefore each morphism $\phi_{\alpha\beta}^\#$ is obtained by base change
from some morphism $B_{\beta,j}\to B_{\alpha,j}$ defined over $A_j$ for
$j\ge i$. Since there are only finitely many overlaps, we can again choose some maximal $i$ so that all gluing morphisms come from $A_i$. The
morphisms glue the affines $U_{\alpha,i}=\Spec(B_{\alpha,i})$
to a supercurve $X_i/A_i$ whose base change to $A$ is $X$.
\end{proof}
\begin{prop} \label{athree: schlessinger}
The superstack $\cc$ satisfies condition (A3) of Theorem \ref{theorem: introduction super artin statement}.
\end{prop}

\begin{proof} We use the notation from Theorem \ref{theorem: introduction super artin statement}, Condition (A3).

Set $C=A' \times_A B$. 
Let $J\subset A$ denote the odd nilpotent ideal and set $A_0:=A_{\bos}=A/J$ $A_0'= A'/J_{A'}$ etc. We may assume that the full reduction of each of the superalgebras is equal to the bosonic reduction, since this doesn't change the argument. 
Since $J$ is nilpotent, $J^{N+1}=0$, and define
\[
A_k:=A/J^{k+1}, \qquad 0\le k\le N.
\]
Similarly, for $J_{A'}\subset A'$ and $J_B\subset B$, and set
\[
A_k':=A'/J_{A'}^{k+1}, \qquad B_k:=B/J_B^{k+1}, \qquad C_k:= C/J_C^{k+1}. 
\]
We have induced maps
\[ \phi_k: \ov{\cc}(C_k) \to \ov{\cc}(A_k') \times \ov{\cc}(B_k), \qquad 0\le k \le N.\]
Let $(Y,W)$ be in $\ov{\cc}(A') \times \ov{\cc}(B)$ 
For $k=0$, all inputs are purely bosonic, so $\phi_0$ is equal to the map
\[ \ov{\cc}_{\bos}(C_0) \to \ov{\cc}_{\bos}(A_0') \times \ov{\cc}_{\bos}(B_0)\]
where $\cc_{\bos}$ is the restriction of $\cc$ to the category of ordinary schemes. Lemma \ref{lemma: supercurves correspondence with picard} identifies $\cc_{\bos}$ as the relative Picard stack $\on{Pic}_{\cc_g/\mm_g}$ of the universal curve $\mc{C}_g$ over $\mm_g$. This is well-known to be an algebraic stack, and thus satisfies the Artin condition \textbf{(S1)(a)}. In particular, $\phi_0$ is surjective, and there exists a $Z_0 \in \ov{\cc}(C_0)$ such that $Y_0 \cong Z_0 \times_{C_0}A_0', W_0 \cong Z_0 \times_{C_0}B_0$. 

We now argue by induction. Suppose for $k$, $Z_k$ is a lift over $C_k$, with pullbacks $Y_k=Z_k\times_{C_k}A_k'$ and $W_k=Z_k\times_{C_k}B_k$. The pair $(Y_{k+1},W_{k+1})$ determines a class in
\[
H^1(Y_k,T_{Y_k/A_k'}\otimes_{A_k'} J_{A'}^k/J_{A'}^{k+1})
\oplus
H^1(W_k,T_{W_k/B_k}\otimes_{B_k} J_B^k/J_B^{k+1}),
\]
while the set of lifts of $Z_k$ to $C_{k+1}$ is a torsor under
\[
H^1(Z_k,T_{Z_k/C_k}\otimes_{C_k} J_C^k/J_C^{k+1}).
\]
Now $J_C^k/J_C^{k+1}$ is annihilated by $J_C$, hence is naturally a $C_0$-module, so there are canonical identifications
\[
T_{Z_k/C_k}\otimes_{C_k} J_C^k/J_C^{k+1}
\cong
T_{Z_0/C_0}\otimes_{C_0} J_C^k/J_C^{k+1},
\]
and similarly
\[
T_{Y_k/A_k'}\otimes_{A_k'} J_{A'}^k/J_{A'}^{k+1}
\cong
T_{Y_0/A_0'}\otimes_{A_0'} J_{A'}^k/J_{A'}^{k+1},
\qquad
T_{W_k/B_k}\otimes_{B_k} J_B^k/J_B^{k+1}
\cong
T_{W_0/B_0}\otimes_{B_0} J_B^k/J_B^{k+1}.
\]
Thus \eqref{intro: schless sonea} is identified with the map
\[
H^1(Z_0,T_{Z_0/C_0}\otimes_{C_0} J_C^k/J_C^{k+1})
\to
H^1(Y_0,T_{Y_0/A_0'}\otimes_{A_0'} J_{A'}^k/J_{A'}^{k+1})
\oplus
H^1(W_0,T_{W_0/B_0}\otimes_{B_0} J_B^k/J_B^{k+1}).
\]
for all $k$.
Since
\[
J_C^k/J_C^{k+1}
=
(J_{A'}^k/J_{A'}^{k+1})\times_{J_A^k/J_A^{k+1}}(J_B^k/J_B^{k+1}),
\]
the projections to $J_{A'}^k/J_{A'}^{k+1}$ and $J_B^k/J_B^{k+1}$ are surjective, so this map is surjective whenever
\[
T_{Z_0/C_0}\otimes_{C_0} A_0' \to T_{Y_0/A_0'},
\qquad
T_{Z_0/C_0}\otimes_{C_0} B_0 \to T_{W_0/B_0}
\]
are surjective, since then the induced map on $H^1$ is surjective because $Y_0$ and $W_0$ are even dimension $1$ so $H^2=0$. To finish the argument, since everything is smooth, the relative tangent sheaves are locally free, and thus are compatible with base change. In particular, the above maps are isomorphisms.

For condition \textbf{(S1)(b)}: Since $M$ is finite over the full reduction $A_0$ of $A$, from \eqref{reduce to the reduction} we have
\begin{equation} \label{isomorphism over the reduction}
\mct_{X/T}\otimes_A M \cong \mct_{X_0/A_0} \otimes_{A_0} M .
\end{equation}
Furthermore, since $A'\to A$ is nilpotent, the superrings $A'[M]$ and $A[M]$ have the same full reduction $A_0$, and the assumption that $X'$ is a lift of $X$ over $A'$ implies
\[
X'\times_A A_0 \cong X\times_A A_0 .
\]
Thus $D_X(M)$ and $D_{X'}(M)$ are both identified with
\[
H^1(X_0,\mct_{X_0/A_0}\otimes_{A_0} M)^+,
\qquad
X_0=X\times_A A_0,
\]
by Lemma \ref{lemma:defs of supercurves} and \eqref{isomorphism over the reduction}.

For condition \textbf{(S2)}: $D_{X_0}(M)$ is identified with the same group above. This group is a finite $A_0 \plus$--module because $X_0$ is proper over $A_0$ and $\mct_{X_0/A_0}\otimes_{A_0} M$ is coherent.
\end{proof}

\begin{prop}
\label{proposition: effectivity supercurves}
The superstack $\cc$ satisfies condition (A4) of Theorem \ref{theorem: introduction super artin statement}.
\end{prop}

\begin{proof}
We use the notation of \ref{proj notation}.

Let $(X_j, X_i \to X_j)$ be a projective system of strongly projective supercurves in $\varprojlim_j \cc(T_j)$. We must show that there exists a strongly projective supercurve $X/T$ such that
\[
\prjpull X \cong X_j
\]
for all $j$. Equivalently, the formal superscheme $\mcx=\varprojlim X_j$ over $\mct=\on{Spf}(R)$ is algebraizable.

By Lemma \ref{lemma: serre vanishing}, there exists a non-zero integer $n$ such that the system
\[
\mclbull = (\mcl_j, \mcl_j \to \mcl_i),
\qquad
\mcl_j := \duaxjtj^{\otimes n}
\]
is strongly relatively very ample and satisfies $R^1\pist \mclbull=0$.

Set $\mce_j := \pist \mcl_j$. For each $j \ge i$, the transition map $\mce_j \to \mce_i$ is the adjoint of the base change morphism
\[
\prjipull \mce_j \lgr \mce_i,
\]
all of which are isomorphisms since $R^1\pist \mclbull=0$. In particular,  $\mcebull=(\mce_j,\mce_j\to \mce_i)$ is a coherent system on $\mct=\on{Spf}(R)$ in the sense of \cite[8.1.4]{fantechi2005fundamental}.

Hence $\mcebull$ determines a coherent sheaf $E$ on $\mct=\on{Spf}(R)$ such that $\prjpull E \cong \mce_j$. Since $\mct$ is the formal completion of $T=\Spec(R)$ along its closed point, the super Grothendieck existence theorem, \cite[Theorem 7.6]{faroogh2019existence}, gives a unique coherent sheaf $\mce$ on $T$ such that $\prjpull \mce \cong \mce_j$ and these identifications recover the transition maps of $\mcebull$.

By construction,
\[
\pp(\mce)\times_T T_j \cong \pp(\mce_j),
\]
so $\pp(\mce)$ algebraizes the formal projective superscheme $\wh{\pp}(\mcebull)$ determined by the system $\pp(\mce_j)$. The closed immersions
\[
X_j \hookrightarrow \pp(\mce_j)
\]
defined by $\mcl_j$ are compatible with the transition maps, and therefore define a closed formal subscheme $\mcx \subset \wh{\pp}(\mcebull)$. Its ideal sheaf corresponds to the coherent system
\[
\ibull = (I_j, I_j \to I_i),
\qquad
I_j := \ker(\oo_{\pp(\mce_j)} \to \iota_*\oo_{X_j}).
\]
Applying Grothendieck existence again, there exists a unique ideal sheaf $I \subset \oo_{\pp(\mce)}$ such that $\prjpull I \cong I_j$. Let
\[
X \subset \pp(\mce)
\]
be the associated closed subscheme. Then $\prjpull X \cong X_j$ for all $j$.

It remains to show that $X/T$ is a strongly projective supercurve. Properness is immediate since $X$ is a closed subscheme of the proper superscheme $\pp(\mce)$. To prove flatness, note that the compatible system of closed subschemes
\[
X_j \subset \pp(\mce_j)=\pp(\mce)\times_T T_j
\]
defines an element of
\[
\varprojlim_j \Hilb(\pp(\mce)/T)(T_j),
\]
where $\Hilb(\pp(\mce)/T)$ is the Hilbert functor of closed subschemes of $\pp(\mce)$ that are flat and proper over the base. Since this functor is representable by a superscheme \cite[Theorem 4.3]{bruzzo2023notes}, it is also an algebraic superstack, and thus  effective by Theorem \ref{theorem: introduction super artin statement}. Thus there exists a closed subscheme
\[
X' \subset \pp(\mce),
\]
flat and proper over $T$, such that $\prjpull X' \cong X_j$ for all $j$. Since also $\prjpull X \cong X_j$ for all $j$, uniqueness in Grothendieck existence implies $X \cong X'$. Hence $X$ is flat over $T$.

Since $X$ is flat over $T$, smoothness and relative dimension $(1|1)$ may be checked on geometric fibers. Sine $T=\Spec(R)$ is the Spec of a local superalgebra $(R, \maxid)$, all geometric fibers identify with the central fiber $X_0/T_0$, which is smooth of dimension $(1|1)$ by assumption. Finally, $X/T$ is strongly projective because this property is stable under nilpotent thickenings.
\end{proof}

\begin{prop}
\label{prop: supercurves satisfy coherent obstruction theory, condition afive}   The superstack $\cc$ satisfies condition (A5) of Theorem \ref{theorem: introduction super artin statement}.
    
\end{prop}

\begin{proof}  The obstruction moduli is the zero module since there are no obstructions to deforming strongly projective supercurves over square-zero extensions. 

\end{proof}

\begin{prop}
\label{prop: supercurves satisfy constructibility, condition asix}
The superstack $\cc$ satisfies condition (A6) of Theorem \ref{theorem: introduction super artin statement}.
\end{prop}

\begin{proof} Let $A$ be a superalgebra, and recall $A_0= A/\frak{N}_A$ denotes the full reduction of $A$. Set $T=\Spec(A)$,  $T_0=\Spec(A_0)$, and $X_0= X \times_A A_0$. 
\smallskip

Injectivity of \eqref{construct ob} is immediate.

For bijectivity of \eqref{construct d}:  Assume that $T_0=\Spec(A_0)$ is integral, so that $T_0=S$. 
By (S1)b  of Proposition \ref{athree: schlessinger}, we can replace $D_X$ with $D_{X_0}$ since the two are canonically identified. 

For every closed point $s\in T_0$, we  have canonical identifications
\begin{align*}
D_{X_0}(A_0^{1|1}) & = H^1(X_0,\mct_{X_0/T_0} \otimes_{A_0} A_0^{1|1}) ^{+} \cong H^1(X_0, \mct_{X_0/A_0}) \\ 
D_{X_0}(k(s)^{1|1}) & = H^1(X_0,\mct_{X_0/T_0} \otimes_{A_0} k(s)^{1|1}) ^{+} \cong H^1(X_0, \mct_{X_0/A_0} \otimes_{A_0} k(s)).
\end{align*}
Under these identifications, the map in \eqref{construct d} is the base change morphism
\[
\phi_s : H^1(X_0,\mct_{X_0/T_0}) \otimes k(s) \longrightarrow H^1(X_0,\mct_{X_0/T_0} \otimes_{A_0} k(s)).
\]

The map $X_0 \to  T_0$ is proper, and $\mct_{X_0/A_0}$ is flat over $T_0$ since $X_0$ is smooth over $A_0$ and $\mct_{X_0/A_0}$ is locally free, henceby \cite[Theorem 3.9]{bruzzo2023notes}, the map 
\[ T_0 \to \z \times \z, \quad
s \longmapsto \dim_{k(s)}H^1(X_s,\mct_{X_0/A_0} \otimes_{A_0} k(s))
\]
is upper semi-continuous. 
Let $V\subset T_0$ be the open subset on which the above map
attains its minimal value in both even and odd dimension. Since $S$ is integral, super Grauert's theorem \cite[Theorem 3.11]{bruzzo2023notes}) implies that $\phi_s$ is an isomorphism for every closed point $s\in V$.
\end{proof}

    \begin{prop} \label{prop: supercurves satisfy aseven, comp with completion}
        The superstack $\cc$ satisfies condition (A7) of Theorem \ref{theorem: introduction super artin statement}.
        \end{prop}
        
        \begin{proof} 
        We have a canonical identification
        \[
        D_{X_0}(M)\cong H^1(X_0,\mct_{X_0/A_0}\otimes_{A_0} M)^{+}
        \]
        For simplicity, we write $A$ for $A_0$ and $X$ for $X_0$.

        Since $M$ is finitely generated over $A$, 
        \[ \varprojlim M/\maxid^nM\cong M\otimes_A\wh{A}, \] 
        and so we are left to show bijectivity of the map: 
        \begin{equation}\label{dcomp}
        H^1(X,\mct_{X/A}\otimes_A M)^{+}\otimes_A\wh{A}
        \;\longrightarrow\;
        H^1\bigl(X,\mct_{X/A}\otimes_A(M\otimes_A\wh{A})\bigr)^{+}.
        \end{equation}
        Since $X$ is separated over $A$, and $\mct_{X/A}$ is coherent, \v{C}ech cohomology compute the sheaf cohomology, and \v{C}ech cohomology of any coherent sheaf $\ff$ on $X/A$ commutes with flat base change along $A\to\wh{A}$, i.e.,
        \[
        H^i(\check{C}^{\bullet}(\ff\otimes_A\wh{A})\bigr)
        \cong
        H^i\bigl(\check{C}^{\bullet}(\ff)\bigr)\otimes_A\wh{A}.
        \]
        \end{proof}
\begin{prop}
\label{prop: supercurves satisfy aeight}
The superstack $\cc$ satisfies condition (A8) of
Theorem~\ref{theorem: introduction super artin statement}.
\end{prop}

\begin{proof}
Let $e\colon S\to T$ be the given \'etale morphism and set
$Y:=X\times_T S$ with projection $\pi_S\colon Y\to S$.

We first identify the deformation spaces. By definition,
\[
D_X(M)=H^1(X,\mct_{X/A}\otimes_A M)^{+} ,
\qquad
D_Y(M\otimes_A B)
=H^1\bigl(Y,(\mct_{Y/B}\otimes_B (M\otimes_A B))\bigr)^+.
\]

Since $Y\to X$ is \'etale, the induced morphism on relative tangent sheaves
$\mct_{Y/B}\to \mct_{X/A}\otimes_{\ox} \oy$ is an isomorphism. Since $Y$ is the base change of $X$, the right hand side is isomorphic to $\mct_{X/A} \otimes_A B$. 

Hence
\[
D_Y(M\otimes_A B)
=
H^1\bigl(Y,((\mct_{X/A}\otimes_A M)\otimes_A B)^{+}\bigr).
\]

Set $\ff:=\mct_{X/A}\otimes_A M$. Using cohomology and base change, we have a canonical isomorphism
\[
e^*R^1\pi_*\ff \xrightarrow{\sim} R^1\pi_{S*}(\pr_X^*\ff),
\]
where $\pr_X \colon Y=X\times_T S \to X$ is the projection. Taking global sections and using
\[
H^1(X,\ff)=H^0(T,R^1\pi_*\ff),
\qquad
H^1(Y,\pr_X^*\ff)=H^0(S,R^1\pi_{S*}(\pr_X^*\ff)),
\]
together with the fact that for any quasi-coherent sheaf $\mc{G}$ on $T$,
\[
H^0(S,e^*\mc{G}) \cong H^0(T,\mc{G})\otimes_A B,
\]
it follows that
\[
H^1(X,\ff)\otimes_A B
\cong
H^1(Y,\pr_X^*\ff).
\]
Then, combining the above identifications, we find:
\[
D_X(M)\otimes_A B \;\cong\; D_Y(M\otimes_A B).
\]

\end{proof}

\section{Super Riemann surfaces} \label{section: moduli of super Riemann surfaces}

We assume $g \ge 2$. 
\smallskip

A genus $g$ \emph{super Riemann surface} is a genus $g$ supercurve $X$ over $k$ equipped with a \emph{superconformal structure}, that is, a rank $(0|1)$ locally free subsheaf $\mc{D} \subset \mc{T}_X$ of the tangent sheaf which is maximally non-integrable in the sense that the induced map via the supercommutator $\bracket$ followed by the projection, 
\begin{equation} \label{beta comp}
\mc{D}^{\otimes 2} \lra{[\ , \ ]} \tbold_X \to \tbold_X / \mc{D},
\end{equation}
 is an isomorphism of locally free sheaves on $X$ \footnote{
    It may be surprising that the composition $\beta$ is $\ox$-linear, since $\bracket$ is not. But, by definition,
    \begin{equation} \label{ox linearity}
    [X, fY]= X(f)\, Y + (-1)^{|X||f|} f[X,Y],
    \end{equation}
    so the failure of $\ox$-linearity is precisely the term $X(f)\, Y$. If $X,Y$ both lie in $\dd$, then $X(f)\, Y$ also lies in $\dd$, and hence vanishes after composing with the quotient map $T_X \to T_X/\dd$.}
    .
Many properties of super Riemann surfaces are proved using the short exact sequence, 
\begin{equation}
    \label{sesofdist}
    0 \lra{{}} \mc{D} \lra{{}} T_X \lra{{}} \dd^{\otimes 2} \lra{{}} 0.
\end{equation}
where $\beta$ is the composition in \eqref{beta comp}.
%  \footnote{The composition $\beta$ is  always $\ox$-linear even though the supercommutator is not. Indeed, from the formula
%     \begin{equation} \label{ox linearity}
%     [X, fY]= X(f)\, Y + (-1)^{|X||f|} f[X,Y],
%     \end{equation}
%     we see that the failure for the supercommuator to be $\ox$-linear is the term $X(f)\, Y$. If $X,Y$ both lie in $\dd$, then $X(f)\, Y$ also lies in $\dd$, and hence vanishes after composing with the quotient map $T_X \to T_X/\dd$.}

        \begin{definition}
            A \emph{family of super Riemann surfaces over a superscheme $T$} is a family of supercurves $X\to T$ equipped with a relative superconformal structure $\dd\subset\tbold_{X/T}$.
            An \emph{isomorphism of super Riemann surfaces}, called a \emph{superconformal isomorphism} is a $T$-linear isomorphism of supercurves $\phi:  X \lgr Y$ such that  
            \[
                d\phi(\dd) = \phipull\dd
                \]
                as subsheaves of $\mct_{X/T}$:
              here $d\phi: \mct_{X/T} \lgr \phi^* \mct_{Y/T}$ denotes the induced isomorphism on relative tangent sheaves. 
          
            \end{definition}

\medskip

A family of \emph{spin curves} over an ordinary base scheme $T$ is a pair $(C/T,L)$ where $C/T$ is an ordinary family of curves over $T$ and $L$ is a line bundle on $C$ with a fixed isomorphism  $L^{\otimes 2} \cong \Omega_{C/T}^1$. 
\begin{lemma}
    \label{lemma: correspondence with spin curves} Families of super Riemann surfaces over purely bosonic base schemes are in one-to-one correspondence with families of spin curves. 
\end{lemma}
\begin{proof}
This fact is well-known but we include it for completeness. 

Let $(X/T, \dd)$ be a super Riemann surface over a purely bosonic base $T$. The bosonic reduction $C=\xbos$ of $X$ is an ordinary family of curves over $T$ by definition of supercurve. The superconformal structure on $X$ determines a spin structure on $C/T$ as follows: Since $T$ is purely bosonic, $\jj^2=0$, and $\jj$ is a locally free sheaf on $C$ of rank $(0|1)$. Furthermore, the restriction of \eqref{sesofdist} to $C$ is the following short exact sequence of locally free sheaves on $C$:
\[ 0 \to \dd \vert_C \to \mct_{X/T} \vert_C = \mct_{C/T} \oplus \jj\inv \to \dd^{\otimes 2} \vert_C \to 0.\]
That $\mct_{X/T} \vert_C = \mct_{C/T} \oplus \jj \inv$ follows from the exactness of the tangent sequence associated to $C \hra X$ and by comparing parity of the terms in it.  

The terms of the above sequence are of rank $(0|1), (1|1)$ and $(1|0)$, respectively, and since each map in the sequence is parity-preserving, they induce isomorphisms
\[ \dd^{\otimes 2} \vert_C \cong \mct_{C/T}, \qquad \jj \inv \cong \dd \vert_C.\]
From which it follows that $\jj$, or more precisely $L=\Pi \jj$, is a spin structure on $C$.

For the other direction, let $(C/T,L)$ be a spin curve. Then the superscheme
\[ X=(|C|, \ox \oplus \Pi L) \]
is a supercurve admitting a morphism $s: X \to C$ induced by $\oc \to \ox$, and one can check that $s^*L \inv$ is a superconformal structure on $X$.
\end{proof}

\begin{lemma} \label{srs are strong projective}
    Super Riemann surfaces are strongly projective for all genus $g \ge 2$.
\end{lemma}

\begin{proof}
Let $(X/T, \dd)$ be a super Riemann surface of genus $g \ge 2$, and set $C=\xbos$. By Lemma \ref{theorem: degree condition for strongly projective}, to prove that $X/T$ is strongly projective, it is enough to show 
\[ \deg(\omega_{X/T} \vert_C) =g-1, \qquad \duaxt:=\Ber(\omegaxt) \]
since this is positive for all $g \ge 2$. 
The degree condition is on the bosonic reduction, so we may assume that the base $T$ is bosonic, and apply the correspondence with spin curves to identify $\jj$ as a spin structure on $C$. Since the degree of any spin structure is $g-1$ (and $\deg(\Pi \jj)=\deg(\jj)$)  from \eqref{degree formula} we have
\[ \deg(\duaxt \vert_C)= (2g-2)-(g-1)=g-1,\]
which is positive for all $g \ge 2$.
\end{proof}

    %   The bosonic reduction of $\mfr$ is the moduli stack of spin curves, $\mc{S}\mc{M}_g$.  A \emph{family of spin curves} over an ordinary scheme $T$ is a smooth and proper morphism $C\to T$ together with a line bundle $L$ on $C$ and an isomorphism $L^{\otimes 2}\cong \Omega^1_{C/T}$.
      
    %   A super Riemann surface $(X/T,\dd)$ over an ordinary scheme $T$ determines a spin curve $(C/T,L)$.  Indeed, the bosonic reduction of $X/T$ is a family of curves $C/T$, and the bosonic reduction of $\dd$,
    %   \[
    %   \dd_{\bos}:=\dd\otimes_{\ox}(\ox/\jj),
    %   \]
    %   is a rank $(0|1)$ line bundle on $C$ equipped with a canonical isomorphism $\dd_{\bos}^{-2}\cong \Omega^1_{C/T}$ induced by \eqref{beta comp}.  Setting $L:=\dd_{\bos}\inv$ yields the associated family of spin curves.
      
    %   Conversely, a spin curve $(C/T,L)$ determines a super Riemann surface $(X/T,\dd)$ by taking
    %   \[
    %   X=\bigl(|C|,\ \ox=\oc\oplus \Pi L\bigr),
    %   \]
    %   with structure morphism induced by $\oc\to\ox$, and defining $\dd\inv$ to be the pullback of $\Pi L$ along the canonical map $X\to C$.
\subsection{Deformation theory of super Riemann surfaces}
   Let $A' \to A$ be a square-zero extension of superalgebras with kernel $M=\ker(A' \to A)$, $M^2=0$, a finite $A_0=A/\frak{N}_A$-module. Set $T=\Spec(A)$, $T'=\Spec(A')$ and $T_0=\Spec(A_0)$.

      A \emph{superconformal deformation of a super Riemann surface $(X/T, \dd)$ over $T'$} is a super Riemann surface $(X'/T', \dd')$ together with a superconformal isomorphism
      \[
      X' \times_{T'} T \cong X
      \]
      over $T$.  Superconformal deformations of $(X/T, \dd)$ over $T'$ are controlled by the sheaf of \emph{relative superconformal vector fields} $\mc{A}_{X/T}$. This sheaf admits a canonical isomorphism 
      \begin{equation} \label{good for a}
      \mc{A}_{X/T} \cong \omega_{X/T}^{\otimes -2},
      \end{equation}
      of sheaves of complex super vector spaces. In particular,  $\mc{A}_{X/T}$ inherits the properties of $\duaxt$.

\begin{prop} \label{prop: defs of srs}

There are no obstructions to deforming super Riemann surfaces over square-zero extensions, and the set of isomorphism classes of superconformal deformations $\sDef_{X/T}(T')$ is a torsor for the group
\[
H^1(X, \mc{A}_{X/T} \otimes_{A} M)^{+}.
\]
\end{prop}

\begin{proof}
Since $X$ is smooth, it admits an \'etale cover by smooth affine superschemes isomorphic to
\[
Y=\Spec A[z,\theta].
\]
We may assume the cover is superconformal in the sense that the pullback of $\dd$ to $Y$ is generated by the odd vector field
\begin{equation} \label{standard scf structure form}
D_{\theta}=\partial_{\theta}+\theta\partial_z .
\end{equation}
and the subsheaves generated by $D_{\theta}$ agree on all overlaps as subsheaves of the relative tangent sheaf.  

Every deformation $Y'$ of such the \'etale chart $Y$ is isomorphic to
\[
\Spec A'[z,\theta]
\]
with superconformal structure generated by $D_{\theta}$. Any choice of such an isomorphism compatible with the superconformal structures is called a \emph{trivialization}.

On overlaps of the \'etale cover, these trivializations glue to a deformation of $X$ if and only if they compose to a superconformal automorphism of $\Spec A'[z, \theta]$ ( the composition of the transition functions  $\sigma$ sends $D_{\theta}$ to $F D_{\theta'}$ where $F$ is an even invertible function and $\theta'=\sigma(\theta)$),
and satisfy the cocycle condition on triple overlaps.

It is well-known that superconformal automorphisms of the overlaps correspond to sections of the sheaf of superconformal vector fields
\[
\mc{A}_{X/T}\subset \mct_{X/T}.
\]
The failure of the automorphisms to satisfy the cocycle condition therefore defines a class in
\[
H^2(X,\mc{A}_{X/T}\otimes_A M)^{+},
\]
and the set of isomorphism classes of superconformal deformations is a pseudo--torsor for
\[
H^1(X,\mc{A}_{X/T}\otimes_A M)^{+}.
\]

Since $M$ is assumed to be finite over $A_0$, by \eqref{isomorphism over the reduction} applied to the subsheaf $\mc{A} \subset \mct$, we have
\[ \mc{A}_{X/T} \otimes_A M \cong \mc{A}_{X_0/T_0} \otimes_{A_0} M.\]
Since $T_0=\Spec(A_0)$ is bosonic, splits into a direct sum of coherent sheaves on the bosonic reduction $C=\xbos$. Indeed, for any sheaf of $\oxnot$-modules, the identity $\oxnot=\oc \oplus \jj$ (as sheaves of modules over $\oc$) implies
\[ \ff \cong \ff \vert_C \oplus (\ff \vert_C \otimes_{\oc} \jj), \qquad \ff \vert_C =\ff \otimes_{\ox} \oc\]
 Thus $H^2$ splits into a direct sum of $H^2$s of the components of $\ff$. Since $C$ has dimension $1$, both component vanish, and thus $H^2(X,\mc{A}_{X_0/T_0}\otimes_{A_0} M)=0$. Hence there are no obstructions to deforming a super Riemann surface over square-zero extensions, and the set of isomorphism classes of superconformal deformations is a torsor for the above group by Theorem \ref{from Hart a}. 

\end{proof}

      \subsection{Moduli of super Riemann surfaces is an algebraic superstack} \label{section: moduli of super riemann surfaces is superstack} 

     We assume that the genus satisfies $g \ge 2$.

Let $\mfr$ be the category fibered in groupoids over $\sS$ defined as follows. For $T\in\sS$, the fiber $\mfr(T)$ is the groupoid whose objects are families of genus $g$ super Riemann surfaces over $T$ and whose morphisms are superconformal isomorphisms. Pullback in $\mfr$ is given by fiber product together with pullback of the superconformal structure.

Concretely, if $(X,\dd)\in\mfr(T)$ and $S\to T$ is a morphism in $\sS$, set $X_S:=X\times_T S$ and let $p\colon X_S\to X$ be the canonical projection. The pulled--back distribution $\dd_S$ is defined as $\ppull\dd$, viewed as a subsheaf of $\mct_{X_S/S}$ via the canonical isomorphism
\[
\ppull \mct_{X/T}\cong \mct_{X_S/S}.
\]
      Then $(X_S,\dd_S)\in\mfr(S)$.

\begin{prop}\label{prop: srs is superstack}
$\mfr$ is an \'etale superstack over $\sS$.
\end{prop}

\begin{proof}

\noindent\emph{(i) Descent for isomorphisms.}
Let $X,Y \in \mfr(U)$ and let $\{U_i \to U\}$ be an \'etale cover. Suppose we are given superconformal isomorphisms
\[
\alpha_i : \pripull X \;\lgr\; \pripull Y
\]
compatible on overlaps. Since isomorphisms of superschemes  satisfy \'etale descent and glue to an isomorphism of supercurves by Proposition \ref{proposition: supercurves is a superstack}, the $\alpha_i$ glue to a unique isomorphism $\alpha : X \to Y$ of supercurves over $U$.

It remains to check that $\alpha$ is superconformal. This holds if and only if
\[
d\alpha(\dd_X) = \alpha^*\dd_Y
\]
as subsheaves of $\alpha\pull \mct_{Y/U}$. Since subsheaves of quasi-coherent sheaves satisfy \'etale descent and the two subsheaves become equal after pullback along the \'etale cover $\pripull X \to X$ (because the $\alpha_i$ are superconformal), they must already coincide on $X$.
\smallskip

\noindent\emph{(ii) Descent for objects.}
Let $\{U_i \to U\}$ be an \'etale cover, and suppose we are given super Riemann surfaces $(X_i/U_i,\dd_i)$ together with superconformal isomorphisms
\[
\alpha_{ij}\colon \prijpull X_i \;\lgr\; \prjipull X_j
\]
satisfying the cocycle condition.

Since super Riemann surfaces are strongly projective by Lemma \ref{srs are strong projective}, Proposition~\ref{proposition: supercurves is a superstack} implies that there exists a unique supercurve $X/U$ such that $\pripull X \cong X_i$ for all $i$. 

Because the $\alpha_{ij}$ are superconformal, the distributions $\dd_i$ agree on overlaps under the identifications $\prijpull X_i \cong \prjipull X_j$. Hence the $\dd_i$ define compatible subsheaves of the tangent sheaf on the overlaps of the \'etale cover $\pripull X \to X$, and therefore glue by \'etale descent to a subsheaf
\[
\dd \subset \mct_{X/U}.
\]

It remains to check that $\dd$ is superconformal. The bracket morphism
\[
\dd^{\otimes 2} \lra{\bracket} \mct_{X/U} \to \mct_{X/U}/\dd
\]
pulls back to the corresponding morphisms for the $(X_i,\dd_i)$, where it is an isomorphism. Since isomorphisms of quasi-coherent sheaves can be checked \'etale locally, the bracket morphism is an isomorphism, and thus $\dd$ is a superconformal structure. 

\end{proof}

      \noindent \textbf{A1.}  Super Riemann surfaces are strongly projective by Lemma \ref{srs are strong projective}, so by Proposition \ref{proposition: rep of diagonal for supercurves}, there exists a Hilbert superscheme
      \[
      \hh:=\mathbb{H}(X\times_T X'/T)
      \]
      parametrizing closed subschemes of $X\times_T X'$ that are proper and flat over $T$.
      Let $U\subset\hh$ be the open subscheme from Proposition \ref{proposition: rep of diagonal for supercurves} representing isomorphisms of the underlying supercurves, forgetting the superconformal structure.
      Denote by $\phi\colon X_U\to X'_U$ the universal isomorphism over $U$.
      The locus where $\phi$ is superconformal is the vanishing locus of the morphism
      \begin{equation}\label{srs isos}
      \dd \hookrightarrow T_{X_U/U} \xrightarrow{d\phi} \phi^*T_{X'_U/U} \twoheadrightarrow \phi^*(T_{X'_U/U}/\dd').
      \end{equation}
      Since \eqref{srs isos} is a morphism of locally free sheaves ($\phipull(\mct'/\dd')$ is locally free because $\dd'$ is by definition of superconformal structure a direct summand of $\mct'$), hence flat over $U$, so its zero locus is a closed subscheme of $U$ by  \cite[Lemma 6.1.a]{felder2020moduli}. 
      \vspace{10mm}

      \noindent \textbf{A2.}  The same arguments from Proposition \ref{atwo: colimit preserving} show there exists $\lambda$ such that $X\cong X_{\lambda} \times_{A_{\lambda}} A$ as supercurves over $A$. The superconformal structure $\dd_{\lambda}$ on $X_{\lambda}$ pulls back to a superconformal structure $\dd_X$ on $X$. Since a supercurve admits at most one superconformal structure, $\dd_X = \dd$.
      \vspace{10mm}

      \noindent \textbf{A3.} We can use the same arguments as in Proposition \ref{athree: schlessinger} after making the following adjustments for (S1)a: By Lemma \ref{lemma: correspondence with spin curves}, the bosonic reduction of $\mfr$ is the moduli of genus $g$ spin curves. This is well-known to be an algebraic (in fact, DM) stack, so the same arguments then go through after replacing $H^1(X, \mct_{X_0/T_0} \otimes_{A_0} M)^{+}$ with  $H^1(X, \mc{A}_{X_0/T_0} \otimes_{A_0} M)^{+}$, and noting that the sheaf of superconformal vector fields $\mc{A}$ is compatible with base change through \eqref{good for a}.  For (S1)b and (S2), we can  use the same arguments as in Proposition \ref{athree: schlessinger}, after making the same replacement. 
      \vspace{10mm}

      \noindent \textbf{A4.}
      Let $\xbull=(X_j, X_i \to X_j)$ be a projective system of super Riemann surfaces in $\varprojlim \mfr(T_j)$. The superconformal structure on $\xbull$ is the projective system 
      \[ \dbull=(\dd_j, \dd_j \to \dd_i)\]
with transition maps the adjoints of $\prjipull \dd_j \cong \dd_i$.

Since the supercurves underlying super Riemann surfaces are strongly projective, Lemma \ref{srs are strong projective},  Proposition \ref{proposition: effectivity supercurves} applies, and there exists a unique supercurve $X$ over $T$ algebraizing the formal superscheme $\mcx$ over $\mct$ determined by $\xbull$. This means that $\varprojlim (X \times_T T_j) \cong \mcx$ as formal superschemes over $\mct=\on{Spf}(R)$.
Let 
\[ \mctbull = (\mct_{X_j/T_j}, \mct_{X_j/T_j} \to \mct_{X_i/T_i} ) \]
denote the system of relative tangent sheaves on $X_j/T_j$ where the transition maps are the standard ones \eqref{standard trans}. 

The systems $\mctbull$ and $\dbull$ are clearly coherent: the transition maps are linear, and induce the canonical isomorphisms $\prjipull \dd_j \cong \dd_i$ and $\prjipull \mct_{X_j/T_j} \cong \mct_{X_i/T_i}$. Furthermore, the inclusions $\dd_j \to \tanxjtj$ are compatible with the transition maps, and thus define a monomorphism of coherent systems $\dbull \to \mctbull$.
Since $X/T$ algebraizes the formal superscheme $\mcx/\mct$, from the super Grothendieck existence theorem we get a unique morphism of sheaves $\dd \to \tanxt$ pulling back to $\dd_j \to \tanxjtj$ for all $j$. That $\tanxt$ is the relative tangent sheaf on $X/T$ follows from uniqueness: Indeed, both $\tanxt$ and the relative tangent sheaf pullback to $\tanxjtj$, so the two must be isomorphic.

Our main goal now is to show that $\dd$ is a superconformal structure on $X$. 

To show that $\dd$ is locally free of rank $(0|1)$, we argue locally on $X$. Since each $\dd_j$ is locally free of rank $(0|1)$, and an affine open cover of $X_0$ induces compatible affine opens on $X_j$, and $X$,  we have a compatible system of isomorphism $\oo_{X_j}^{0|1} \lgr \dd_j$ on each affine of the cover.  By the super Grothendieck existence theorem, this compatible system algebraizes uniquely to an isomorphism
\[
\ox^{0|1} \xrightarrow{\sim} \dd
\]
on the corresponding open of $X$. Hence $\dd$ is locally free of rank $(0|1)$.

To prove the remaining properties of a superconformal structure, we will need the following lemma.

\begin{lemma} \label{inside}
Suppose $\phi:\ff\to\mcg$ is a morphism of coherent sheaves on $X$ such that
\[
\prjpull\phi:\prjpull\ff\to\prjpull\mcg
\]
is injective (resp.\ surjective, bijective) for all $j$. Then $\phi$ is injective (resp.\ surjective, bijective).
\end{lemma}

\begin{proof}
We first prove injectivity. Let $K=\ker(\phi)$. For each $j$ there is a canonical surjective morphism
\[
\prjpull K \longrightarrow \ker(\prjpull\phi)
= \ker(\ff_j\to\mcg_j),
\qquad
\ff_j=\prjpull\ff,\;\mcg_j=\prjpull\mcg .
\]
Since $\prjpull\phi$ is injective, the right-hand side vanishes, hence
\[
\prjpull K=0 \quad \text{for all } j.
\]
where:
\[
\prjpull K
= K\otimes_{\ox}\mc{O}_{X_j}
= K\otimes_R R/\maxid^{j+1}
\cong K/\maxid^{j+1}K .
\]
Thus $\prjpull K=0$ implies $K=\maxid^{j+1}K $, and since this holds for all $j$,
\[
K \subset \bigcap_{j\ge0}\maxid^{j+1}K .
\]
Because $K$ is coherent and $R$ is Noetherian, the Krull intersection theorem implies
\[
\bigcap_{j\ge0}\maxid^{j+1}K = 0,
\]
and hence $K=0$. This proves injectivity.
The proofs for surjectivity and bijectivity are analogous. 
\end{proof}

We can now apply the lemma to prove that $\dd$ is a superconformal structure.

By Lemma \ref{inside} the morphism $\dd \to \tanxt$ is injective because its pullbacks to $X_j$ are injective. Furthermore, this implies that the short exact sequence $0 \to \dd \to \tanxt \to \tanxt/\dd \to 0$ remains short exact after pullback, and thus
\[ \prjpull(\tanxt/\dd) \cong \prjpull(\tanxt)/\prjpull(\dd) \cong \tanxjtj/\dd_j.  \]
The same argument using Grothendieck existence shows that $\tanxt/\dd$ is locally free of rank $(1|0)$. 

Finally, define $\beta$ to be the composition $\dd^{\otimes 2} \xrightarrow{\bracket} \tanxt \to \tanxt/\dd$. Then 
\[ \prjpull(\beta)= \beta_j, \qquad \beta_j: \dd_j^{\otimes 2} \xrightarrow{\bracket} \tanxjtj \to \tanxjtj/\dd_j \]
Since $\dd_j$ is superconformal, $\beta_j$ is an isomorphism, and so $\beta$ is an isomorphism by Lemma \ref{inside}. In particular,  $\dd$ is a superconformal structure on $X/T$. 
      \vspace{10mm}

\noindent \textbf{A5.} By Proposition \ref{prop: defs of srs}, there are no obstructions to deforming super Riemann surfaces over square-zero extensions and so there are no non-trivial obstruction modules. 
\vspace{10mm}

\noindent \textbf{A6.}
The proof is the same as in
Propositions \ref{prop: supercurves satisfy constructibility, condition asix}
-- \ref{prop: supercurves satisfy aeight} after replacing $\tanxt$ with
$\mc{A}_{X/T}$.
Indeed, in those arguments the only properties of the sheaf used are that it is
coherent and flat over the base in order to apply the cohomology
base--change theorem and super Grauert’s theorem \cite{bruzzo2023notes}, and $\mc{A}_{X/T}$ via the 
isomorphism \eqref{good for a}. 

\begin{remark} \label{remark about freeness} For the argument it is enough that the deformation sheaf (in this case $\mc{A}$) is coherent. Indeed, the constructibility
condition is imposed on dense open subsets of the connected components of the base, and every coherent sheaf on $X$ is generically flat over
every integral subscheme of the base.
\end{remark}
\vspace{10mm}

\noindent \textbf{A7, A8.}
The proofs are identical to those of
Propositions \ref{prop: supercurves satisfy aseven, comp with completion}
and \ref{prop: supercurves satisfy aeight} after replacing $\mct_{X/T}$
with $\mc{A}_{X/T}$.

\section{Stable supercurves} \label{section: stable supercurves}

Let $k$ be an algebraically closed field, and let $X$ be a proper, Cohen--Macaulay superscheme of finite type over $k$. 
We assume that the bosonic reduction 
\[
C=(|X|,\ox/\jj)
\]
is a one-dimensional scheme, and that $X$ has superdimension $1|1$ at every smooth point: a point $q \in X$ is \emph{smooth} if there exists an isomorphism
\[
\widehat{\mc{O}}_{X,q} \cong k[[z,\theta]],
\]
in which case $\widehat{\mc{O}}_{X,q,\bos} \cong k[[z]]$.

Let $q \in X$ be a closed point. We denote by 
\[
\widehat{\mc{O}}_{X,q} = \varprojlim \mc{O}_{X,q}/\mathfrak{m}_q^n
\]
the completion of the local superring at $q$ with respect to its maximal ideal.

We say that $q$ is a \emph{Ramond node} if there exists an isomorphism
\[
\widehat{\mc{O}}_{X,q} \cong k[[x,y,\theta]]/(xy),
\]
and a \emph{Neveu--Schwarz (NS) node} if there exists an isomorphism
\[
\widehat{\mc{O}}_{X,q} \cong k[[x,y,\theta_1,\theta_2]]/(xy,\,x\theta_2,\,y\theta_1,\,\theta_1\theta_2).
\]
In both cases, the bosonic reduction of $\widehat{\mc{O}}_{X,q}$ is isomorphic to $k[[x,y]]/(xy)$.

By a \emph{node} on $X$ we mean either a Ramond node or an NS node.

We say that $X$ is a \emph{nodal supercurve} if every point of $X$ is either smooth or a node, and the bosonic reduction
\[
C=(|X|,\ox/\jj)
\]
is an ordinary nodal curve.

By definition, the bosonic reduction $C$ is obtained by quotienting $\ox$ by the nilpotent ideal $\jj$, so the canonical morphism
\[
X \to C
\]
induces a homeomorphism on underlying topological spaces. In particular, $|X|=|C|$.

Moreover, at a point $q \in X$, the completed local rings satisfy
\[
\widehat{\mc{O}}_{C,q} \cong \widehat{\mc{O}}_{X,q}/\jj_q.
\]
It follows from the explicit local models that $q$ is smooth (respectively a node) on $X$ if and only if its bosonic reduction is smooth (respectively a node) on $C$. Therefore, the set of singular points of $X$ coincides with the singular locus of $C$. 

In particular, if $\Sigma \subset X$ denotes the singular locus of $X$, then
\[
|\Sigma| = \mathrm{Sing}(C),
\qquad 
|X \setminus \Sigma| = C_{\mathrm{sm}}.
\]

Since $X$ has superdimension $(1|1)$, the ideal $\jj$ is locally free of rank $(0|1)$ and satisfies $\jj^2=0$, so that
\[
\ox = \oc \oplus \jj .
\]
The inclusion $\oc \hookrightarrow \ox$ induces a morphism
\[
X \to C,
\]
whose pre-composition with the canonical closed immersion $C \hra X$
is the identity on $C$.

This projection is used to define the normalization of $X$. Specifically, since $\jj^2=0$,  $\ox$ and $\oc=\ox/\jj$ have the same integral closure, and the  normalization $\tilx$ of $X$ satisfies $\tilx \cong X \times_C \tilde{C}$. The normalization map $\nu: \tilx \to X$ is  the top horizontal arrow in the cartesian diagram
\[
\begin{tikzcd}
\tilde{X} \arrow[r, "\nu"] \arrow[d] & X \arrow[d] \\
\tilde{C} \arrow[r] & C
\end{tikzcd}
\]
where $\tilde{C} \to C$ is the normalization of $C$.

Since $|X|=|C|$, the schemes $X$ and $C$ have the same set of irreducible components, which we denote by $X_i$ and $C_i$, respectively. The genus $g_i$ of a component $X_i$ is defined to be the genus of the corresponding component $C_i$.

A point $\tilde{q} \in \tilx$ is called \emph{special} if $q=\nu(\tilde{q})$ is a node.

\begin{definition}\label{definition: stable supercurve}
A connected nodal supercurve $X/k$ is \emph{stable} if
\[
2g_i - 2 + \#\{\text{special points in } \tilx_i\} > 0
\]
for every irreducible component $X_i \subset X$.
\end{definition}

\begin{definition}
\label{definition: prestable supercurve}
A prestable supercurve over a superscheme $T$ is a proper, flat, relatively Cohen--Macaulay morphism $X \to T$
such that for every geometric point $\ovt \to T$ the fiber $\xovt$ is a connected nodal supercurve over $k(\ovt)$. 
In this case, the bosonic reduction
\[
\pibos: C \to \tbos, \qquad C:=\xbos
\]
is automatically a family of ordinary prestable curves.

A family $X/T$ of prestable supercurves is called \emph{stable} if for every geometric point $\ovt \to T$ the fiber $X_{\ovt}$ is a stable supercurve over $k(\ovt)$.
\end{definition}

Historically, stable curves are preferred to prestable ones because stable curves have finite automorphism groups, which implies that their moduli stack is Deligne--Mumford. If one instead works with prestable curves, the resulting moduli stack generally has non-finite automorphism groups and is therefore only algebraic. 
For ease of reference we will work with stable supercurves, though the  arguments presented here apply equally to prestable supercurves.

\begin{prop} \label{prop: torsion free sheaves of rank 1}
Stable supercurves $X/T$ over a purely bosonic base scheme $T$ are in one-to-one correspondence with pairs $(C/T,L)$, where $C/T$ is a family of ordinary stable curves and $L$ is a torsion-free sheaf of rank $1$ on $C$.
\end{prop}
\begin{proof}
For the forward direction, let $C=\xbos$ and set $L=\Pi\jj$. By definition the bosonic reduction $C$ is a family of stable curves over $T$, so it remains to show that $\jj$ is torsion-free of rank $(0|1)$ on $C$. The same then holds for $L=\Pi\jj$, since the parity-reversal functor only changes the parity of the rank.

Since $T$ is purely bosonic, the odd ideal sheaf $\jj\subset\ox$ satisfies $\jj^2=0$. This is obviously true on the smooth locus, and can be deduced from the definition of the nodes on the singular locus. Hence the $\ox$-module structure on $\jj$ factors through $\oc=\ox/\jj$, so $\jj$ is naturally an $\oc$-module. Because $\jj$ is coherent over $\ox$, it is also coherent over $\oc$.

Let $U\subset C$ denote the smooth locus of $C$. Since $C$ is nodal, $U$ is dense. Its preimage $U'\subset X$ is the smooth locus of $X$, and $U$ is the bosonic reduction of $U'$. The ideal sheaf of the embedding $U\hookrightarrow U'$ is $\jj|_U$, and since $\jj|_U^2=0$, it is also the conormal sheaf of this embedding. Because $U$ and $U'$ are smooth and $\on{sdim}(U')=(1|1)$ while $\dim(U)=1$, it follows that $\jj|_U$ is locally free of rank $(0|1)$.

It remains to prove torsion-freeness. Since $\ox=\oc\oplus\jj$, and $X$ and $C$ are defined to be $T$-flat, $\jj$ is flat over $T$. Thus torsion-freeness may be checked fiberwise, so we may assume $T=\Spec k$. Let $p\in X$ and set $A=\mc{O}_{X,p}$, $J=\jj_p$, and $B=A/J=\mc{O}_{C,p}$. Since $X$ is Cohen--Macaulay of even dimension $1$, the local ring $A$ has depth $1$ and therefore admits a non-zerodivisor. If $g\in B$ is a non-zerodivisor and $g\theta=0$ for some $\theta\in J$, then for any lift $\tilde g\in A$ we have $\tilde g\theta=0$. As $\tilde g$ is a non-zerodivisor in $A$, it follows that $\theta=0$. Hence $J$ is torsion-free over $B$, and therefore $\jj$ is torsion-free as an $\oc$-module.

Conversely, suppose $(C/T,L)$ is a family of stable curves with $L$ a torsion-free sheaf of rank $1$ on $C$. We claim that the superscheme
\[
X=(|C|,\oc\oplus\Pi L)
\]
is a stable supercurve over $T$.

Let $p$ be a smooth point of $C$. Then $\wh{B}:=\wh{\mc{O}}_{C,p}\cong k[[z]]$. Since any torsion-free sheaf of rank $1$ on a smooth curve is locally free, we have $\wh{L}_p=L_p\otimes\wh{B}\cong\wh{B}$. Choosing a generator $\theta$ of $\Pi\wh{L}_p$, we obtain
\[
\wh{B}\oplus\Pi\wh{L}_p \cong k[[z,\theta]] .
\]

Now let $p$ be a node, and set $\wh{B}=\wh{\mc{O}}_{C,p}\cong k[[x,y]]/(xy)$. A torsion-free $\wh{B}$-module of rank $1$ is isomorphic to either $\wh{B}$, or  ideal $(x,y)\subset \wh{B}$ which under the normalization $\wh{B}\hookrightarrow k[[x]]\oplus k[[y]]$ is $(x)\oplus(y)\subset k[[x]]\oplus k[[y]]$, \cite[Proposition 1.2.8]{orecchia2014torsion}). 

In the first case we again choose a generator $\theta$ and obtain
\[
\wh{B}\oplus\Pi\wh{L}_p \cong k[[x,y,\theta]]/(xy).
\]
In the second case, $\wh{L}_p \cong (x)\oplus(y) \subset k[[x]]\oplus k[[y]]$, write $\theta_1=\Pi x$ and $\theta_2=\Pi y$ for the generators, noting the relations $x\,\Pi y=0$, $y\,\Pi x=0$, $(\Pi x)(\Pi y)=0$, and implying
\[
\theta_1\theta_2=0,\qquad x\theta_2=0,\qquad y\theta_1=0,
\]
and therefore
\[
\wh{B}\oplus\Pi\wh{L}_p
\cong
k[[x,y,\theta_1,\theta_2]]/(xy,x\theta_2,y\theta_1,\theta_1\theta_2).
\]

Flatness of $X$ over $T$ follows because both $\oc$ and $\Pi L$ are flat over $T$, and properness can be checked on the bosonic reduction $C$.
\end{proof}

\begin{lemma}\label{lemma: stable are projective}
Every stable supercurve $X/T$ is projective.
\end{lemma}

\begin{proof}
Let $C/S$ denote the bosonic reduction: $C=\xbos$ and $S=\tbos$.  
The canonical embedding $C \to X$ induced by $\ox \to \ox/\jj$, $\jj^{n+1}=0$, factors as a finite sequence of square-zero extensions
\[
\ox \to \ox/\jj^n \to \cdots \to \ox/\jj^2 \to \ox/\jj ,
\]
where $\ker(\ox/\jj^i \to \ox/\jj^{i-1})=\jj^{i-1}/\jj^{i}$.  
Setting $X_i=(|X|,\ox/\jj^{i+1})$, this yields closed immersions
\[
X_0 \hra X_1 \to \cdots \to X .
\]

Let $\mcl_0:=\omega_{C/S}$. We lift $\mcl_0$ inductively along the above square-zero extensions.  
The obstruction to lifting $\mcl_0$ to $X_1$ lies in
\[
H^2(|C|,\jj/\jj^2),
\]
which vanishes since $C$ has relative dimension $1$. Hence $\mcl_0$ lifts to a line bundle $\mcl_1$ on $X_1$.  

Assuming $\mcl_k$ is defined on $X_k$, the obstruction to lifting it to $X_{k+1}$ lies in
\[
H^2(|C|,\jj^k/\jj^{k+1}),
\]
which again vanishes: Indeed, $\jj^k/\jj^{k+1}$ is a $\oc=\ox/\jj$-module since $\jj \cdot (\jj^k/\jj^{k+1})=0$, and $H^2$ of any quasi--coherent sheaf on $C$ vanishes.  By induction we obtain a line bundle $\mcl$ on $X$ with
\[
\mcl|_C=\mcl_0 .
\]

For any $n\ge3$ we then have
\[
\mcl^{\otimes n}|_C \cong \omega_{C/S}^{\otimes n}.
\]
Since $\omega_{C/S}^{\otimes n}$ is strongly relatively very ample for $n\ge3$, Proposition~\ref{propatwo} implies that $\mcl^{\otimes n}$ is strongly relatively very ample on $X/T$, hence $X/T$ is projective.
\end{proof}

The lemma will be used to establish effectivity of the moduli of stable supercurves. 
\ref{proj notation}. 
 \begin{lemma} \label{lemma: proj system of line on stable} Using the notation from Section \ref{proj notation}, let $\xbull=(X_j, X_i \to X_j)$ be a projective system of stable supercurves over $(T_j, T_i \to T_j)$. There exists a projective system $\mclbull=(\mcl_j, \mcl_j \to \mcl_i)$ of strongly relatively very ample line bundles on  $\xbull$ satisfying $R^1 \pist \mclbull=0$.   
 \end{lemma}

 \begin{proof} Let $C$ denote the bosonic reduction of $X_0$, and let $S=\tbos$. Fix an integer $n >>0$ such that $L=\omega_{C/S}^{\otimes n}$ is strongly relatively very ample and satisfies $R^1 \pist \omega_{C/S}^{\otimes n}=0$ The proof of the previous proposition shows that we can lift $L$ to a line bundle $\mcl_0$ on $X_0$. Since $X_0 \to X_1$ is nilpotent, the same arguments again can be used to lift $\mcl_0$ to a line bundle $\mcl_1$ on $X_1$, and by induction we can lift $\mcl_{j-1}$ on $X_{j-1}$ to a line bundle $\mcl_j$ on $X_j$. This gives rise to a projective system of  
 line bundles $\mclbull=(\mcl_j, \mcl_j \to \mcl_i)$ on $\xbull$ whose transition maps are the adjoints of the base change isomorphisms $\prjipull \mcl_j \cong \mcl_i$.  
 
Each line bundle $\mcl_j$ is strongly relatively ample since that property can be checked after restriction along the nilpotent extension $X_0 \to X$, where $\mcl_j \vert_{X_0} \cong \mcl_0$ and $\mcl_0$ is strongly very ample because its restriction to $C$ is (Proposition \ref{propatwo}). 
If necessary, increase $n$ so that $R^1 \pist \mcl_0 =0$. Then the same arguments as in Lemma \ref{lemma: serre vanishing}, show that $R^1 \pist \mcl_j=0$ for all $j$
 
 \end{proof}

\subsection{Deformation theory of stable supercurves} \label{section on deformations of stabel supercurves}

Let $A' \to A$ be a square zero extension of superalgebras with kernel $M=\ker(A' \to A)$, $M^2=0$, a finite $A_0=A/\frak{N}_A$-module. Set $T=\Spec(A)$, $T'=\Spec(A')$, and $T_0=\Spec(A_0)$.  
\smallskip

A \emph{deformation} of a stable supercurve $X/T$ over $T'$ is a stable supercurve $X'/T'$ such that
\[
X'\times_{T'}T \;\cong\; X .
\]
over $T$. 
\begin{prop} \label{deformations of stabel supercurves}
Let $X/T$ be a stable supercurve. There are no obstructions to deforming stable
supercurves over square--zero extensions. Moreover, fixing a deformation
$X_1'/T'$, there is a short exact sequence
\begin{equation}\label{ses of ss}
0 \to H^1(X,\mct_{X/T}\otimes_{A}M)^{+}
\to \on{Def}_{X/T}(T')
\to H^0(X,\mct_{X/T}^1 \otimes_{A}M)^{+}
\to 0 .
\end{equation}
\end{prop}

\begin{proof} Once we
prove that stable supercurves have unobstructed deformations, the theorem follows from Theorem~\ref{from Hart a} and \eqref{isomorphism over the reduction}, we have
\[
H^2(X,\mct_{X/T}\otimes_{A}M) \cong H^2(X_0, \mct_{X_0/T_0} \otimes_{A_0} M)  =0 ,
\]
since $T_0$ is purely bosonic, $\mct_{X_0/T_0}$ is a direct sum of coherent sheaves on $C=\xbos$ (see also proof of Lemma \ref{lemma:defs of supercurves}), and thus the above group vanishes since $C$ has dimension $1$. 
\smallskip

We first treat the case $T=\Spec(k)$. Let $\Sigma\subset X$ denote the singular locus and set $U=X\setminus\Sigma$. The restriction $U$ is smooth,
hence deformations of $U$ are unobstructed. Since deformations satisfy
\'etale descent, it suffices to show that \'etale neighborhoods of the
nodes admit unobstructed deformations with compatible gluing on overlaps. 

By Proposition \ref{etalenbhd}, there exists an \'etale cover
$\{X_i\to X\}$ such that the image of each $X_i$ is either contained in
$U$ or intersects $\Sigma$ in exactly one point, and each such chart containing a node is \'etale over a standard local model
\[
Y=\Spec(R)
\]
where

\[
R=
\begin{cases}
k[x,y,\theta]/(xy) & \text{(Ramond node)},\\[4pt]
k[x,y,\theta_1,\theta_2]/(xy,x\theta_2,y\theta_1,\theta_1\theta_2)
& \text{(NS node)}.
\end{cases}
\]The charts contained
in $U$ are unobstructed, so we restrict attention to those intersecting
$\Sigma$.

Unobstructedness of these models is proved in
Lemma~\ref{lemma: miniversal node}.  Hence every $X_i$ admits
deformations over arbitrary local Artin bases.

We now glue these local deformations. Let $(A,\maxid)$ be a local Artin
superalgebra with residue field $k$. Choose deformations $X_i'/A$ of
each chart $X_i$. The chosen \'etale cover from \ref{etalenbhd} has only self--overlaps,
so the induced deformations automatically agree on overlaps.

The only remaining obstruction is the Čech $2$--cocycle measuring
failure of the transition isomorphisms to satisfy the cocycle condition.
This obstruction lies in
\[
H^2(X,\mct_{X/T}\otimes_A M)^{+},
\]
which we already showed vanishes for dimension reason (and the fact that $M$ is assumed to be finite $A_0=k$-module).

For arbitrary $T$, we prove that $X/T$ deforms over $T'$ in  Lemma \ref{lemma: over Tprime}. 
\end{proof}

\begin{lemma}\label{lemma: miniversal node}
The Ramond and NS node models admit miniversal deformation families.

For the Ramond node let
\[
R=k[x,y,\theta]/(xy),
\qquad
B=k[t,\eta],
\]
and define
\[
\mcy=\Spec B[x,y,\theta]/(xy-t-\eta\theta).
\]

For the NS node let
\[
R=k[x,y,\theta_1,\theta_2]/(xy,x\theta_2,y\theta_1,\theta_1\theta_2),
\qquad
B=k[t_1,t_2,\eta_1,\eta_2],
\]
and define
\[
\mcy=\Spec B[x,y,\theta_1,\theta_2]/
(xy-t_1t_2-\eta_1\theta_1-\eta_2\theta_2,
x\theta_2-t_1\theta_1,
y\theta_1-t_2\theta_2,
\theta_1\theta_2).
\]

In both cases $\mcy\to\Spec(B)$ is flat and satisfies
\[
Y \cong \mcy\times_B k .
\]
Moreover, for every local Artin $k$--superalgebra $(A,\maxid)$ and
every deformation $Y_A$ of $Y$ over $A$, there exists a morphism
$B\to A$ such that
\[
Y_A \cong \mcy\times_B A .
\]
\end{lemma}

\begin{proof}
The computation of the tangent modules $T^1(R/k,R)$ (see
Appendix~\ref{computation of ti for the nodes}) shows that the parameters deformations over square--zero extensions are determined by precisely the variables appearing
in the algebras $B$ above. Hence every deformation over
$A_1=A/\maxid^2$ arises uniquely from a morphism $B\to A_1$.

Given such a deformation over $A_1$, any lift of the corresponding
map $B\to A_1$ to $B\to A/\maxid^3$ produces a deformation lifting
the previous one. Since the parameters of $B$ may be chosen
arbitrarily modulo powers of $\maxid$, such lifts always exist.

Repeating this construction  inductively along the composite of $A \to k$ into square--zero extension shows that
every deformation over $A$ arises by base change from $\mcy$.
\end{proof}

\begin{lemma}\label{lemma: over Tprime}
Let $X/S$ be a stable supercurve, let $x\in X$ be a node, and let $s\in S$ be its image. 
Then there exist an \'etale neighborhood $S'\to S$ of $s$, a morphism $f:S'\to \Spec(B)$, and an \'etale neighborhood $X'\to X\times_S S'$ of $x$ such that
\[
X' \cong \mcy \times_B S',
\]
where $\mcy\to\Spec(B)$ is the miniversal deformation corresponding to the type of node $x$.

In particular, if $T\hra T'$ is a square--zero extension with kernel $M$, then the \'etale morphism $S'\to S$ admits a unique lift $S''\to T'$, and for any morphism $f':S''\to\Spec(B)$ satisfying $f'\equiv f \mod M$, the superscheme $\mcy \times_B S''$ is a deformation of $X'\cong X\times_S S'$ over $T'$.
\end{lemma}

\begin{proof}
The statement is \'etale local on $S$ near $s$, so after replacing $S$ by an \'etale neighborhood of $s$ we may assume that $S=\Spec(R^h)$, where $R^h=\mc{O}_{S,s}^h$ is the henselization of the local ring of $S$ at $s$. Write $\widehat{R^h}$ for its completion. By the defining property of the miniversal deformation, after base change to $\Spec(\widehat{R^h})$ there exist a morphism $\widehat{f}:\Spec(\widehat{R^h})\to\Spec(B)$ and an \'etale neighborhood $\widehat{X}'\to X\times_S \Spec(\widehat{R^h})$ of $x$ such that $\widehat{X}'\cong \mcy\times_B \Spec(\widehat{R^h})$.

Since the moduli of stable supercurves is colimit preserving by Proposition~\ref{prop: colimit preserving for stable supercurves}, we may assume that $S$ is of finite type over $k$. Then the map $(R^h)^+\to(\widehat{R^h})^+$ is geometrically regular \cite{grothendieck1965elements}. Hence, by Lemma 7.30 of \cite{faroogh2019existence}, the superring $\widehat{R^h}$ is a filtered direct limit of smooth $R^h$-superalgebras, say $\widehat{R^h}\cong\varinjlim_i R_i$. Since $B$ is finitely generated over $k$, the morphism $B\to\widehat{R^h}$ factors through some $R_i$. Since $\widehat{X}'$ is finite over $\Spec(\widehat{R^h})$, it comes from an \'etale neighborhood $X_i'\to X\times_S \Spec(R_i)$ of $x$ together with an isomorphism $X_i'\cong \mcy\times_B \Spec(R_i)$.

The composite $R_i\to \widehat{R^h}\to k(s)$ defines a $k(s)$-point of the smooth $R^h$-superscheme $\Spec(R_i)$. Since $R^h$ is henselian and $R_i$ is smooth over $R^h$, this point lifts to an $R^h$-point, equivalently, the structure map $R^h\to R_i$ admits a section $\sigma:R_i\to R^h$. Pulling back $X_i'$ along $\sigma$, we get an \'etale neighborhood $X'\to X\times_S \Spec(R^h)$ of $x$ such that $X'\cong \mcy\times_B \Spec(R^h)$. Then by definition of henselization, there exists an \'etale neighborhood $S' \to S$ of $s$ such that $X' \cong \mcy \times_B S'$. 

For the final statement, let $T\hra T'$ be a square--zero extension with kernel $M$. Since $S'\to S$ is \'etale, there is a unique lift $S''\to T'$ of the induced morphism $T\to S'$. Any morphism $f':S''\to\Spec(B)$ satisfying $f'\equiv f\mod M$ defines the pullback $\mcy\times_B S''$, whose restriction to $T$ is $\mcy\times_B S'\cong X'$, so $\mcy\times_B S''$ is a deformation of $X'\cong X\times_S S'$ over $T'$.
\end{proof}

\subsection{Moduli of stable supercurves is an algebraic superstack} \label{section: moduli of stable supercurves is an alg stack}

Let $\ovscg$ denote the category over $\sS$ whose fiber over a superscheme
$T$ is the groupoid of genus $g$ stable supercurves over $T$. The pullback
in $\ovscg$ is given by fiber products, and the cartesian morphisms are
the canonical morphisms induced by these fiber products.

\begin{prop} \label{prop: stable supercurves are a superstack}
$\ovscg$ is an \'etale superstack over $\sS$.
\end{prop}

\begin{proof}
The statement can be proved using the same arguments as in
\cite[Proposition 7.6]{faroogh2019existence}. That proposition is
stated for stable super Riemann surfaces, but the proof applies
verbatim to stable supercurves, since a stable super Riemann surface
is a stable supercurve equipped with additional structure.

We also provide a alternative proof, more aligned with out previous arguments. 
We first treat the case in which the bosonic reduction $\tbos$ has even dimension at least $1$. Since dimension is invariant under \'etale base change, the same holds for each $T_{i,\bos}$. Let $\Sigma_i \subset X_i$ denote the singular locus, let $Y_i:=X_i\setminus \Sigma_i$ be the smooth locus, let $C_i:=(X_i)_{\bos}$, and let $U_i:=Y_{i,\bos}\subset C_i$ be the smooth locus of $C_i$. Since each geometric fiber of $X_i/T_i$ is a nodal supercurve and $T_{i, \bos}$ has positive dimension, the subset $\Sigma_i$ has codimension at least $2$ in $X_i$. Hence, by Hartogs, morphisms and isomorphisms defined on $Y_i$ extend uniquely to $X_i$.

Choose a line bundle $L_i$ on $X_i$ lifting $\omega_{C_i/T_i}^{\otimes n}$ for some fixed $n\ge 3$. We claim that such a lift is unique up to unique isomorphism. Indeed, if $L_i'$ is another lift, then $L_i|_{Y_i}$ and $L_i'|_{Y_i}$ are both lifts of $\omega_{C_i/T_i}^{\otimes n}|_{U_i}$, because the smooth loci of $X_i$ and $C_i$ agree after bosonic reduction. Now $U_i\to T_i$ is affine, and since $Y_i$ is a nilpotent thickening of $U_i$, $Y_i$ is also affine. Choose a factorization of the nilpotent thickening $U_i\hookrightarrow Y_i$ into square-zero extensions. At each step, the set of lifts of a line bundle is a torsor under $H^1$ of the corresponding square-zero ideal sheaf, and these groups vanish because all intermediate thickenings are affine. Thus the lift of $\omega_{C_i/T_i}^{\otimes n}|_{U_i}$ to $Y_i$ is unique up to unique isomorphism, so $L_i|_{Y_i}\cong L_i'|_{Y_i}$. By Hartogs, this extends uniquely to an isomorphism $L_i\cong L_i'$ on $X_i$. This proves the claim.

Now set $X_{ij}:=X_i\times_T T_j$, and similarly define $Y_{ij},U_{ij},C_{ij},L_{ij}$. Since $\alpha_{ij}:X_{ij}\xrightarrow{\sim}X_{ji}$ is an isomorphism of stable supercurves, its bosonic reduction identifies $C_{ij}$ with $C_{ji}$, and therefore $\alpha_{ij}^*L_{ji}$ is again a lift of $\omega_{C_{ij}/T_{ij}}^{\otimes n}$. By the uniqueness just proved, there exists an isomorphism
\[
\phi_{ij}:L_{ij}\xrightarrow{\sim}\alpha_{ij}^*L_{ji}.
\]
We show that the $\phi_{ij}$ can be chosen to satisfy the cocycle condition. On the smooth locus $Y_{ij}$ the lift of $\omega_{C_{ij}/T_{ij}}^{\otimes n}|_{U_{ij}}$ is unique up to unique isomorphism by the affine argument above, so the isomorphisms $\phi_{ij}|_{Y_{ij}}$ are uniquely determined. In particular, on a triple overlap $Y_{ijk}$ the two isomorphisms
\[
L_{ijk}\longrightarrow \alpha_{ij}^*\alpha_{jk}^*L_{ki}
\]
given by $\phi_{ik}|_{Y_{ijk}}$ and by $\alpha_{ij}^*(\phi_{jk})\circ \phi_{ij}|_{Y_{ijk}}$ coincide. Therefore 
\[
(\alpha_{ij}^*\phi_{jk})\circ \phi_{ij}\circ \phi_{ik}^{-1}
\]
is the identity on $Y_{ijk}$. Since this is an automorphism of the line bundle $L_{ijk}$ and $X_{ijk}\setminus Y_{ijk}$ has codimension at least $2$, Hartogs implies that it is the identity on all of $X_{ijk}$. Thus the $\phi_{ij}$ satisfy the cocycle condition on triple overlaps.

We can now repeat the argument of Proposition \ref{proposition: supercurves is a superstack}, setting $\mce_i = \pist \mcl_i$, and get a coherent sheaf $\mce$ on $T$ with $\prjpull \mce \cong \mce_j$, and a closed subscheme $X \subset \pp_T(\mce)$, flat and proper over $T$, with $\prjpull X \cong X_j$. 

The superscheme $X/T$ is a stable supercurve. Indeed, it is proper, flat, and relatively Cohen--Macaulay because these properties are \'etale local on the base, and its geometric fibers are stable supercurves because this is also checked after \'etale base change. The closed singular subschemes $\Sigma_i \subset X_i$, are compatible on overlaps and thus glue to singular locus $\Sigma \subset X$. The defining conditions of the nodes can be checked on geometric fibers. 

This proves descent when $\tbos$ has even dimension at least $1$. The argument for even dimension $0$ is follows from $\on{Isom}$ being a sheaf. 
\end{proof}

\begin{prop}
    \label{proposition: rep of diagonal for prestable supercurves}
    The superstack $\ovscg$ satisfies condition (A1) of Theorem \ref{theorem: introduction super artin statement}.
   
    \end{prop}
    
    \begin{proof} Suppose $X$ and $X'$ are stable supercurves over $T$. By Lemma \ref{lemma: stable are projective}, $X$ and $X'$ are both projective, flat and proper over $T$, and so their product $X \times_T X'$ is also. By Theorem 4.3 in \cite{bruzzo2023notes}, there exists a Hilbert superscheme parameterizing closed subschemes of $X \times_T X'$ which are flat and proper over $T$, and so we can repeat the arguments from  Proposition \ref{proposition: rep of diagonal for supercurves}. 
            \end{proof}

\begin{prop} \label{prop: colimit preserving for stable supercurves} The superstack $\ovscg$ satisfies condition (A2) of Theorem \ref{theorem: introduction super artin statement}.
\end{prop}
\begin{proof} 
We can repeat word for word the proof of Proposition \ref{atwo: colimit preserving} since the main argument used only that a supercurve $X/A$ is proper over $A$, and that $A$ is a Noetherian superalgebra. A stable supercurve $X/A$ is proper by definition, and $A$ is Noetherian since all superalgebras appearing in this paper are assumed to be Noetherian.
\end{proof}

\begin{prop} \label{prop: schlessinger for stable supercurves}
The superstack $\ovscg$ satisfies condition (A3) of Theorem \ref{theorem: introduction super artin statement}.
\end{prop}

\begin{proof}
For (S1a): Proposition \ref{prop: torsion free sheaves of rank 1} implies that the bosonic reduction of $\ovscg$ is the stack of torsion free sheaves of rank $1$ on the moduli of stable curves. This is an algebraic stack, \emph{e.g.} \cite{orecchia2014torsion}. We may therefore apply Proposition \ref{athree: schlessinger} to the two outside terms in \eqref{ses of ss}, and reduce to proving that the following maps are surjective:
\begin{equation} \label{one}
H^1(Z_0,T_{Z_0/C_0}\otimes_{C_0} J^k/J^{k+1}) \to H^1(Y_0,T_{Y_0/A_0} \otimes_{A_0} J_{A'}^k/J_{A'}^{k+1}),
\end{equation}
\begin{equation} \label{two}
H^0(Z_0,T_{Z_0/C_0}^1\otimes_{C_0} J^k/J^{k+1}) \to H^0(Y_0,T_{Y_0/A_0}^1 \otimes_{A_0} J_{A'}^k/J_{A'}^{k+1}),
\end{equation}
\begin{equation} \label{three}
H^1(Z_0,T_{Z_0/C_0}\otimes_{C_0} J^k/J^{k+1}) \to H^1(W_0,T_{W_0/B_0} \otimes_{B_0} J_B^k/J_B^{k+1}),
\end{equation}
\begin{equation} \label{four}
H^0(Z_0,T_{Z_0/C_0}^1\otimes_{C_0} J^k/J^{k+1}) \to H^0(W_0,T_{W_0/B_0}^1 \otimes_{B_0} J_B^k/J_B^{k+1}).
\end{equation}

Consider \eqref{one}; the argument for \eqref{three} is identical. Let
\[
\phi_Y : \prjpull\bigl(T_{Z_0/C_0}\otimes_{C_0} J^k/J^{k+1}\bigr) 
\longrightarrow 
T_{Y_0/A_0}\otimes_{A_0} J_{A'}^k/J_{A'}^{k+1}, \qquad \prj: Y_0 \hra Z_0
\]
be the natural map induced by base change. Since $J^k/J^{k+1} \to J_{A'}^k/J_{A'}^{k+1}$ is surjective and the relative tangent sheaf commutes with base change along the smooth locus (where it is locally free), the map $\phi_Y$ is a surjection over the smooth locus. Thus, $\coker(\phi_Y)$ is supported on the singular locus $\Sigma_{Y_0}$.

Factoring $\phi_Y$ through its image we get a short exact sequences
\[
0 \to \ker(\phi_Y) \to 
\prjpull\bigl(T_{Z_0/C_0}\otimes_{C_0} J^k/J^{k+1}\bigr)
\to \on{image}(\phi_Y) \to 0
\]
and
\[
0 \to \on{image}(\phi_Y) \to 
T_{Y_0/A_0}\otimes_{A_0} J_{A'}^k/J_{A'}^{k+1}
\to \coker(\phi_Y) \to 0.
\]
Since $\Sigma_{Y_0}$ is $0$-dimensional, 
\[
H^1(Y_0,\ker(\phi_Y))=H^1(Y_0,\coker(\phi_Y))=0.
\]
Taking cohomology, we find that the induced map
\[
H^1\bigl(Y_0,\prjpull(T_{Z_0/C_0}\otimes_{C_0} J^k/J^{k+1})\bigr)
\to
H^1\bigl(Y_0,T_{Y_0/A_0}\otimes_{A_0} J_{A'}^k/J_{A'}^{k+1}\bigr)
\]
is surjective. Identifying $H^1(Z_0,\ff) \cong H^1(Y_0,\prjpull \ff)$ for the coherent sheaf $\ff=T_{Z_0/C_0}\otimes_{C_0} J^k/J^{k+1}$, this proves surjectivity of \eqref{one}.

For \eqref{two} and \eqref{four}, the sheaves $T^1_{Z_0/C_0}$, $T^1_{Y_0/A_0}$, and $T^1_{W_0/B_0}$ are supported on their respective singular loci. By \ref{computation of ti for the nodes}, each is a direct sum of copies of the structure sheaf of the singular locus $\osig$: $(2|2)$ copies along the NS locus and $(1|1)$ copies along the Ramond locus. These identifications are compatible with pullback. Since the singular loci of $Y_0$ and $W_0$ are obtained by pullback from that of $Z_0$, the induced maps
\[
\osig \to \oo_{\Sigma_{Y_0}}, \qquad \osig \to \oo_{\Sigma_{W_0}}
\]
are surjective, and hence the corresponding maps on the $T^1$-sheaves are surjective. This proves \eqref{two} and \eqref{four}.

Conditions (S1b) and (S2) follow from Proposition \ref{athree: schlessinger} applied to the two outside terms of \eqref{ses of ss}.
\end{proof}

   \begin{prop} \label{prop: effectivity for stable supercurves}
  The superstack $\ovscg$ satisfies condition (A4) of Theorem \ref{theorem: introduction super artin statement}.
    
\end{prop}

\begin{proof} Using the notation of Theorem \ref{theorem: introduction super artin statement}, let $(X_j/T_j, X_i \to X_j)$ in $\varprojlim \ovscg(T_j)$ be a projective system of stable supercurves over $(T_j, T_i \to T_j)$, and let $\mclbull=(\mcl_j,  \mcl_j \to \mcl_i)$ be a projective system of strongly relatively very ample line bundles satisfying $R^1 \pist \mcl_{\bullet} =0$ as in Lemma \ref{lemma: proj system of line on stable}. 

Let $\mce_j= \pist \mcl_j$, and for each $j \ge i$, define the transition map $\mce_j \to \mce_i$ to be the composition of the adjunction map $\mce_j \to \pr_{ij*} \prijpull \mce_j = \prijpull \mce_j$ with the base change morphism $\prijpull \mce_j \to \mce_i$. 
Since $R^1 \pist \mcl_j =0$, the base change morphism is an isomorphism, and so the projective system $\mcebull=(\mce_j, \mce_j \to \mce_i)$ is a coherent system of $\mc{O}_{T_j}$ modules on the formal affine superscheme $\mct=\on{Spf}(R)$. Since $\mct$ is affine, there exists a unique coherent sheaf $\mce$ on $T=\Spec(R)$ such that $\prjpull \mce \cong \mce_j$. 
We can now repeat the arguments of  Proposition \ref{proposition: effectivity supercurves}, to find a superscheme $X \subset \pp(\mce)$, flat and proper over $T$, such that $\prjpull X \cong X_j$. 

We claim $\xbos$ is an ordinary family of stable curves over $\tbos$. Indeed, the system $(X_j/T_j)$ reduces to a compatible system of ordinary stable curves $C_j:=(X_j)_{\bos}$ over $(T_j)_{\bos}$, and
\[
(X_{\bos}\times_{T_{\bos}} (T_j)_{\bos}) \cong (X_j)_{\bos}=C_j.
\]
The system $(C_j/T_j)$ is an object in $\varprojlim \ov{\mm}_g(T_j)$, where $\ov{\mm}_g$ is the moduli of stable curves. This is famously an algebraic stack (in fact DM), and is therefore effective. Thus, there exists a unique stable curve $C$ over $T$ such that $\prjpull C \cong C_j$. Since $X_{\bos}$ and $C$ are both isomorphic to $C_j$ after pullback by $\prj$, $X_{\bos} \cong C$ by uniqueness. 

\end{proof}

  \begin{prop} \label{proposition: remaining conditions for stable supercurves}
  The superstack $\ovscg$ satisfies condition (A5-A8) of Theorem \ref{theorem: introduction super artin statement}.
    
\end{prop} 
\begin{proof} 
Since there are no obstructions to deformation a stable supercurve over square-zero extensions, there are no non-zero obstructions modules, so (A5) holds automatically.

The deformation functor $\Def_{X/T}(T')$ fits into the exact sequence
\[
0 \to H^1(X,\mct_{X/T}\otimes_{A}M)^+ \to \Def_{X/T}(T') \to H^0(X,\mct^{1}_{X/T}\otimes_{A}M)^+ \to 0,
\]
so to verify conditions (A6)--(A8) it suffices to check them for the two outer modules.

For (A7) and (A8) this follows immediately from the arguments of Propositions \ref{prop: supercurves satisfy aseven, comp with completion} and \ref{prop: supercurves satisfy aeight} applied to the coherent sheaves $\mct_{X/T}$ and $\mct^1_{X/T}$ on the proper morphism $X \to T$, since those proofs only use coherence of the deformation sheaves, and flat and properness of $X$ over $T$.

For (A6) we apply the argument of Proposition \ref{prop: supercurves satisfy aeight}. The proof there uses that the deformation module can be written as $H^1(X_0,\mc{F}\otimes M)^+$ for a sheaf $\mc{F}$ on $X_0$ that is flat over $T_0$. For stable supercurves, $\mct_{X_0/T_0}$ and $\mct^1_{X_0/T_0}$ are not necessarily flat over $T_0$, but they are coherent. This is enough to apply the same arguments, see Remark \ref{remark about freeness}. 
\end{proof}

\section{Stable super Riemann surfaces} \label{section: stable super Riemann surfaces}

We now recall several properties of smooth super Riemann surfaces that will be relevant for the definition of a stable super Riemann surface.

Let $(X/T,\dd)$ be a \emph{smooth} super Riemann surface. Applying $\Ber$ to the short exact sequence
\begin{equation}
\label{dual of ses seq}
0 \to \dd^{-2} \to \Omega_{X/T}^1 \to \dd^{-1} \to 0
\end{equation}
gives a canonical isomorphism
\[
\duaxt := \Ber(\omegaxt) \cong \dd^{-1}.
\]

In particular, we find that the superconformal structure on $X$ induces a surjective $\ox$-linear morphism $\omegaxt \to \duaxt$, and thus also a grading-preserving $\ot$-linear derivation
\[
\delta: \ox \to \duaxt.
\]

A fundamental fact about super Riemann surfaces is the existence of \'etale local coordinates $(z, \theta)$ on $X$, called \emph{superconformal coordinates}, in which $\dd$ is locally generated by the odd vector field: 
\begin{equation} \label{standard form} D_{\theta} = \frac{\partial}{\partial \theta} + \theta \frac{\partial}{\partial z}.\end{equation}
In these superconformal coordinates $(z,\theta)$, the derivation $\delta$ takes the form
\begin{equation}
\label{scf der}
\delta(f)=D_{\theta}(f)[dz | d\theta].
\end{equation}

Conversely, a surjective $\ot$-linear grading-preserving derivation $\delta:\ox\to\duaxt$ arises from a superconformal structure on $X$ if and only if the following conditions hold.

\begin{itemize}

\item[(1)]
$\dd=\ker(\delta)^{\perp}\subset \mct_{X/T}$ is maximally non-integrable in the sense that the composition
\[
\dd\otimes\dd \xrightarrow{\bracket} \mct_{X/T} \to \mct_{X/T}/\dd
\]
is an $\ox$-linear isomorphism.

\item[(2)]
Locally,
\[
\delta(f)=D(f)[D^2/D]^{-1},
\]
where $D$ and $D^2$ are local generators of $\dd$ and $\dd\otimestwo$, respectively, and
\[
[D^2/D]^{-1}:=(D^{-1}\otimes D^2)^{-1}=D\otimes D^{-2}
\]
is a local generator of $\duaxt$ under the identification $\duaxt\cong\dd\otimes\dd^{-2}=\dd^{-1}$.

\end{itemize}

If $\delta$ arises from a superconformal structure, then there exist \'etale local coordinates $(z,\theta)$ on $X$ such that $\delta$ is locally of the form \eqref{scf der} with $D_{\theta}$ replaced by $D$.

\begin{definition}\label{definition: stable super Riemann surface}
A family of \emph{stable super Riemann surfaces} over a superscheme $T$ consists of a pair
\[
(X \to T,\; \delta\!: \ox \to \omega_{X/T}),
\]
where $X \to T$ is a family of stable supercurves and $\delta$ is a grading-preserving $\ot$-linear derivation such that

\begin{enumerate}

\item[(i)]
the restriction of $\delta$ to the smooth locus $U \subset X$ is superconformal;

\item[(ii)]
for every geometric point $\ovt \to T$, the induced morphism
\begin{equation}
\label{geo fiber odd}
\delta^- \colon \mc{O}_{X_t}^- \lgr \omega_{X_t}^-
\end{equation}
is an $\mc{O}_{C_t}$-linear isomorphism.

\end{enumerate}

We refer to $\delta$ as a \emph{quasi superconformal structure} (qscf structure) on $X/T$. An isomorphism of stable super Riemann surfaces $X/T \to Y/T$ is an isomorphism $\phi\!: X \to Y$ of stable supercurves over $T$ such that $\phipull(\delta_Y)=\delta_X$. Such an isomorphism is called a \emph{superconformal isomorphism}.
\end{definition}

A \emph{quasi spin curve} $(C/T, \mcl, b)$ over an ordinary scheme $T$, as in Definition 2.1.2 of \cite{jarvis1998torsion}, is an ordinary family of stable curves $C/T$ together with a pair $(\mcl,b)$ where $\mcl$ is a torsion-free sheaf of rank one and degree $g-1$, and
\[
b:\mcl \otimestwo \to \omega_C
\]
is a morphism of $\oc$-modules such that $b$ is an isomorphism on the open subset where $\mcl$ is locally free, and for each singular point $q$ of $C$ where $\mcl$ is not locally free, the length of the cokernel of $b$ at $q$ is one.
These conditions on $b$ are equivalent to requiring that $b$ induce an isomorphism
\[
\mcl \lgr \Hom_{\oc}(\mcl, \omega_C).
\]

\begin{prop} \label{prop: stable srs and quasi spin curves} There is a one-to-one correspondence between families of stable super Riemann surfaces over an ordinary base scheme $T$ and families of  quasi spin curves over $T$.
    
\end{prop}

\begin{proof} Let $(X/T, \delta)$ be a stable super Riemann surface over an ordinary scheme $T$. By definition, $C=\xbos$ is an ordinary family of stable curves over $T$, and we showed in Proposition \ref{prop: torsion free sheaves of rank 1}, that $\jj \subset \ox$ is a torsion-free sheaf of $\oc$-modules of rank $(0|1)$.  We therefore set
\[
\mcl := \Pi \jj .
\]
We are left to show that $(X/T,\delta)$ determines a morphism
\[
b:\mcl \otimestwo \to \omega_{C/T}
\]
satisfying the conditions of a quasi spin structure. 

The proof is from \cite{deligneletter}.
Let $p: X\to C$ denote the morphism onto $C$ induced by the inclusion $\oc\to\ox= \oc \oplus \Pi L$. Then
\[
\duaxt := p^{!}\omega_{C/T}=\Hom_{\oc}(\ox,\omega_{C/T}).
\]

The assumption that $X/T$ is relatively Cohen--Macaulay, implies $\duaxt$ is flat over $T$, with formation compatible with base change. We have: 
\begin{align*}
\duaxt^+ &= \omega_{C/T}, \\
\duaxt^- &= \Hom_{\oc}(\mcl,\omega_{C/T}).
\end{align*}

The derivation $\delta$ therefore induces morphisms
\begin{align*}
\tildelt^+ &: \oc \to \omega_{C/T} \qquad\text{(a derivation)},\\
\tildelt^- &: \mcl \to \Hom_{\oc}(\mcl,\omega_{C/T}) \qquad\text{($\oc$-linear)} .
\end{align*}

Since $\Pi \mcl = \ox \minus$, the map $\tildelt^-$ is an isomorphism on every geometric fiber. Since both sides are coherent $\oc$-modules compatible with base change, it follows that $\tildelt^-$ is an isomorphism globally. This induces the desired morphism
\[
b:\mcl \otimestwo \to \omega_{C/T}.
\]

\end{proof}

\subsection{Superconformal coordinate charts}

Let $X/T$ be a stable super Riemann surface. A \emph{superconformal deformation} of $X/T$ over a square-zero extension $T \hra T'$ is a stable super Riemann surface $X'/T'$ with a superconformal identification
\[ X \cong X' \times_{T'} T\]
over $T$.

By Theorem 3.1 of \cite{felder2020moduli} (following \cite{deligneletter}), the Ramond and NS models of the two kinds of nodes on a stable supercurve  admit miniversal superconformal deformations. 
The base of the miniversal superconformal deformation for both models is $\Spec k[t]$, and the total space for the Ramond model is 
\[ \mcy=\Spec k[t][x,y, \theta]/(xy-t).\]
For the NS model the total space is 
\[ \mcy=\Spec k[x,y, \theta_1, \theta_2]/(xy-t^2, x \theta_2 - t \theta_1, y \theta_1 - t \theta_2, \theta_1 \theta_2).\]

The canonical qscf structure on $\mcy$ are described in \cite{felder2020moduli}. We provide them here for reference.  
\begin{itemize}

\item \emph{Ramond model.}The node in $\mcy$ is the point $y$ associated to the maximal ideal $(x,y, \theta)$. The complement of the node $y\in \mcy$ is the union of two branches $U_1$ and $U_2$, where $x$ is invertible on $U_1$ and $y$ on $U_2$. The sheaf $\omega_{\mcy/k[t]}$ is free with basis $b$ given by
\[
b=
\begin{cases}
\left[\dfrac{dx}{x}\,\middle|\, d\theta\right] & \text{on } U_1,\\[1em]
-\left[\dfrac{dy}{y}\,\middle|\, d\theta\right] & \text{on } U_2 .
\end{cases}
\]

There is a canonical qscf structure $\delta$ on $\mcy$ given by: 
\[
\delta(f)=
\begin{cases}
(\partial_\theta+\theta x\partial_x)(f)\cdot b & \text{on } U_1,\\
(\partial_\theta-\theta y\partial_y)(f)\cdot b & \text{on } U_2 .
\end{cases}
\]

\item \emph{NS model.}  The node in $\mcy$ is the point $y$ associated to the maximal ideal $(x,y, \theta_1, \theta_2)$. 
The complement of the node is again the union of two branches $U_1$ and $U_2$, on which $x$ and $y$ are respectively invertible.

The generators $\left[dx\mid d\theta_1\right]$ and $\left[dy\mid d\theta_2\right]$ of $\omega_{U_1}$ and $\omega_{U_2}$ extend to sections of $\omega_{Y/A}$ (vanishing on the opposite branch), but they do not generate $\omega_{\mcy/k[t]}$. There is an additional section
\[
s_0=
\begin{cases}
\dfrac{\theta_1}{x}\left[dx\mid d\theta_1\right] & \text{on } U_1,\\[1em]
-\dfrac{\theta_2}{y}\left[dy\mid d\theta_2\right] & \text{on } U_2 .
\end{cases}
\]

There is a canonical qscf structure $\delta$ on $\mcy$ given by: 
\[
\delta(f)=
\begin{cases}
(\partial_{\theta_1}+\theta_1\partial_x)(f)\cdot b & \text{on } U_1,\\
(\partial_{\theta_2}+\theta_2\partial_y)(f)\cdot b & \text{on } U_2 .
\end{cases}
\]
\end{itemize}

The next theorem is really a corollary of the existence of miniversal superconformal deformations of the two kinds of models, the proof is entirely analogous to that of Lemma \ref{lemma: over Tprime}. 

\begin{prop}
    \label{etale structure theorem} Let $(X/S, \delta)$ be a stable super Riemann surface, let $x \in X$ be a node, and let $s \in S$ be its image.  Then there exists an \'etale neighborhood $S' \to S$ of $s$, a morphism
    \[ f: S' \to \Spec k[t],\]
    and an \'etale neighborhood $X' \to X \times_S S'$ of $x$, and a superconformal isomorphism
    \[ X' \cong \mcy \times_{k[t]} S'\]
    where $\mcy \to \Spec k[t]$ is the miniversal superconformal deformation corresponding to the type of node $x$, and $\mcy  \times_{k[t]} S'$ has the induced qscf structure. In particular, by the same reasoning as Lemma \ref{lemma: over Tprime}, there are no obstructions to deforming a stable super Riemann surface over a square--zero extension $S \hra S'$.  
 
    \end{prop}

    \begin{proof} We prove in Section \ref{section: moduli of stable super Riemann surfaces is a superstack} that the moduli of stable super Riemann surfaces is colimit preserving. This allows us to assume that the base $S$ is finite type over $k$, and so we can repeat the arguments of Lemma \ref{lemma: over Tprime}.
        
    \end{proof}

\subsection{Deformation theory of stable super Riemann surfaces}

In this section we describe the deformation theory of stable super
Riemann surfaces. Conceptually, it is obtained from the
deformation theory of the underlying stable supercurve by imposing
a superconformal condition along the nodes. 

Our main result is that this deformation theory is governed by  certain \emph{coherent} subsheaves
\[
\mc{A}^i_{X/T} \subset \mct^i_{X/T}
\]
of the sheaves controlling the deformation theory of the underlying
stable supercurve.

The following theorem is a direct consequence of the definition of $\mc{A}^1_{X/T}$ together with the $T^i$–formalism for deformation theory applied to these subsheaves.
\begin{theorem}\label{scf deformations of stable srs}
Let $T\hra T'$ be a square-zero extension with kernel $M$, and let
$(X/T,\delta)$ be a stable super Riemann surface. 
Upon fixing a superconformal deformation $(X_1'/T',\delta_1')$ of $(X/T, \delta)$---one always exists by Proposition \ref{etale structure theorem} --- we have a short exact sequence
\begin{equation}\label{ses for s srs}
0 \to H^1(X,\mc{A}_{X/T}\otimes_A M)^+
\to \sDef_{X/T}(T')
\to H^0(X,\mc{A}_{X/T}^1\otimes_A M)^+
\to 0
\end{equation}
where $\sDef_{X/T}(T')$ denotes the set of superconformal deformation
classes.
\end{theorem}

We explain the structure of the proof. By Proposition \ref{deformations of stabel supercurves}, the underlying
supercurve $X/T$ has unobstructed deformations and there is an
exact sequence
\[
0 \to H^1(X,\mct_{X/T}\otimes_A M)^+
\to \Def_{X/T}(T')
\to H^0(X,\mct^1_{X/T}\otimes_A M)^+
\to 0 .
\]

A deformation $X'/T'$ of $X/T$ is superconformal if it admits a qscf
structure extending the one on $X$. Such a structure, if it exists,
is unique, so the forgetting map
\[
\sDef_{X/T}(T') \longrightarrow \Def_{X/T}(T'),\qquad
(X'/T',\delta') \mapsto X'/T'
\]
is injective.

Composing this inclusion with the right arrow in the sequence above
gives a map
\begin{equation}\label{map into T1}
\sDef_{X/T}(T') \longrightarrow
H^0(X,\mct^1_{X/T}\otimes_A M)^+ .
\end{equation}

The kernel consists of superconformal deformations locally
isomorphic to the fixed deformation $(X_1'/T',\delta_1')$,
and is naturally identified with
\[
H^1(X,\mc{A}_{X/T}\otimes_A M)^+ .
\]
Thus Theorem~\ref{scf deformations of stable srs} reduces to
constructing a subsheaf
\[
\mc{A}_{X/T}^1 \subset \mct_{X/T}^1
\]
whose global sections describe the image of
\eqref{map into T1}.

\subsection{Definition of $\mc{A}^1$} \label{definition of aone}

Recall from the supercurve deformation theory that the sheaf
$\mct_{X/T}^1$ controls local deformations of the nodes of $X/T$
and is supported on the singular locus $\Sigma\subset X$.
Étale locally along $\Sigma$ the curve is modeled on the superconformal Ramond
and NS node charts described in the previous section. Let
\[
Y=\Spec(R) \xrightarrow{\phi} \Sigma \subset X
\]
be one such chart.

Since $\mct_{X/T}^1$ is compatible with \'etale localization, for such a chart we have
\[
\Gamma(Y,\phi^*\mct_{X/T}^1)\cong T^1(R/A,R\otimes_A M).
\]

The appendix computes these modules explicitly. Writing
\[
N=M\otimes_A R/I,
\]
one obtains
\[
T^1(R/A,R\otimes_A M) \cong
\begin{cases}
N\oplus N\theta & \text{(Ramond node)},\\
N\theta_1\oplus N\theta_2\oplus N\theta_1\oplus N\theta_2
& \text{(NS node)} .
\end{cases}
\]

Thus deformations of the Ramond node are parametrized by pairs
$(a,\alpha)$, while deformations of the NS node are parametrized
by $(a_1,a_2,\alpha_1,\alpha_2)$.

By \cite{deligneletter}, such a deformation is superconformal precisely when
\[
\alpha=0 \qquad\text{(Ramond case)},
\]
and
\[
a_1=a_2,\qquad \alpha_1=\alpha_2=0
\qquad\text{(NS case)}.
\]

These conditions define kernels of linear maps.
For the Ramond node define
\[
\pi_2:T^1(R/A,R\otimes_A M)\to N,\qquad (a,\alpha)\mapsto\alpha
\]
and set
\[
A^1(R/A,P)=\ker(\pi_2).
\]

For the NS node define
\[
\sigma(\alpha_1, \alpha_2, a_1,a_2)
=(\alpha_1,\alpha_2,a_1-a_2,a_2-a_1)
\]
and set
\[
A^1(R/A,R\otimes_A M)=\ker(\sigma).
\]

In both cases $A^1(R/A,R\otimes_A M)^+$ parametrizes superconformal
node deformations.

Since $Y$ is affine these modules define coherent sheaves
$\mc{A}_{Y/T}^1$ satisfying
\[
\Gamma(Y,\mc{A}_{Y/T}^1)=A^1(R/A,R\otimes_A M).
\]
Moreover $R/I=\Gamma(Y,\phi^*\osig)$, so the maps above define
$\oy$--linear morphisms
\[
\phi^*\mct_{X/T}^1 \longrightarrow
M\otimes_A\phi^*\osig
\]
whose kernels are precisely $\mc{A}_{Y/T}^1$.

\begin{prop}
The subsheaves $\{\mc{A}_{Y/T}^1\subset\phi^*\mct_{X/T}^1\}$
agree on overlaps and therefore glue to a coherent subsheaf
\[
\mc{A}_{X/T}^1 \subset \mct_{X/T}^1 .
\]
\end{prop}

\begin{proof}
The étale charts intersect $\Sigma$ in a single point, so their
only overlaps are self--overlaps. Since the defining morphisms
have canonically identified sources and targets, their pullbacks
along the two projections coincide, and the kernels therefore
agree on overlaps.
\end{proof}

\begin{prop}
The definition of $\mc{A}_{X/T}^1$ is independent of the choice of
superconformal coordinates.
\end{prop}

\begin{proof}
Let
\[
\phi\colon Y=\Spec(R)\to \Sigma\subset X
\qquad\text{and}\qquad
\phi'\colon Y=\Spec(R)\to \Sigma\subset X
\]
be two superconformal coordinate charts of the same type, so that
\[
\phi'=\phi\circ \rho
\]
for some superconformal automorphism
\[
\rho\colon (Y,\delta)\to (Y,\delta)
\]
over $T=\Spec(A)$. We must show that the two submodules of
$T^1(R/A,R\otimes_A M)$ obtained from the charts $\phi$ and $\phi'$ coincide.

The automorphism $\rho$ lifts to an action of $T^1(R/A,R\otimes_A M)$. Indeed, let
$Y'/A'$ be a superconformal deformation of $Y$ with superconformal identification denoted by
\[
i\colon Y \xrightarrow{\sim} Y'\times_{A'} A .
\]
Define $\rho\cdot Y'$ to be the same $A'$--superscheme $Y'$, but with  superconformal identification twisted by $\rho$:
\[
i\circ \rho\colon Y \xrightarrow{\sim} Y'\times_{A'} A .
\]
This gives an action of $\on{sAut}(Y,\delta)$ on deformation classes of
$Y$ over $A'$. Using $Y'$ as a base point, we have a canonical identification of deformation
classes of $X/T$, forgetting the qscf structure, with $T^1(R/A,R\otimes_A M)^+$. The $\rho$ action clearly descends to an $R\plus$-linear action on $T^1(R/A,R\otimes_A M)^+$ by superconformal automorphisms $\rho \in \on{sAut}(Y,\delta)$.

To prove the lemma,  it suffices to show that $A^1(R/A,R\otimes_A M)^+\subset T^1(R/A,R\otimes_A M)^+$
is stable under this action. Let $[Y']\in A^1(R/A,R\otimes_A M)^+$, so by
definition $Y'$ admits a qscf structure $\delta'$ extending the given
qscf structure $\delta$ on $Y$. We claim that $\rho\cdot Y'$ is again
superconformal. Indeed, we equip the same total space $Y'$ with the
same qscf structure $\delta'$. With respect to the new identification
of the special fiber, its reduction is
\[
(i\circ \rho)^*(\delta'|_{Y'\times_{A'}A})
=
\rho^*\bigl(i^*(\delta'|_{Y'\times_{A'}A})\bigr)
=
\rho^*(\delta)
=
\delta ,
\]
because $\rho$ is superconformal. Thus, $(Y', i \circ\rho)$ is a superconformal deformation of $Y$.
\end{proof}

\subsection{Moduli of stable super Riemann surfaces is an algebraic superstack} \label{section: moduli of stable super Riemann surfaces is a superstack}

Let $\ovmfr$ be the category over $\sS$ whose fiber over $T \in \sS$ is the groupoid $\ovmfr(T)$ whose objects are families $(X \to T, \delta)$ of stable super Riemann surfaces  and whose morphisms are superconformal isomorphisms.
The pullback in $\ovmfr$ of an object $(X/T, \delta)$ by a morphism $S \to T$ in $\sS$ is the pair
\[ (X_S, \delta_S):= (X \times_T S, \ \ppull \delta: \ppull \ox \cong \oxs \to \ppull \duaxt \cong \duaxs )\]
where $p: X_S \to X$ is the canonical morphism, and both isomorphisms are canonical. 
We claim $\delta_S$ is quasi superconformal: For (i): if $U$ is the smooth locus in $X$, then $U_S$ is smooth in $X_S$. Let $\dd$ denotes the superconformal structure on $U$ determined by $\delta$. Then by definition the derivation $\tildelt$ associated to $\delta$ fits into the short exact sequence 
\[ 0 \to \dd^{-1} \to \Omega_{U/T}^1 \xrightarrow{\tildelt \vert_U} \dd^{-1} \cong \duaxt \to 0\]
Since $\dd$ pulls back to a superconformal structure on $U_S$, the sequence remains exact after pullback: 
\[ 0 \to \ppull \dd^{-1} \to \Omega_{U_S/S}^1 \xrightarrow{\ppull\tildelt \vert_U} \ppull \dd^{-1} \cong \duaxs \vert_{U_S} \to 0\]
where $\delta_S$ is the derivation associated to $\ppull \tildelt \vert_U$.

For (ii): for each geometric point $\ov{s} \to S$ the corresponding geometric fiber of $X_S$ is canonically identified with the geometric fiber of $X$ over the point $\ov{s} \to S \to T$ and $\delta_S$ and $\delta$ restrict to the same derivation under this identification.

\begin{prop}
    \label{prop: stable srs are an etale superstack}
    $\ovmfr$ is an \'etale superstack. 
\end{prop} 

\begin{proof}
    \noindent \emph{Descent for isomorphisms.}  Let $X,Y$ in $\ovmfr(U)$, and let $\{U_i \to U\}$ be an \'etale covering, and let $\alpha_i: \pripull X \lgr \pripull Y$ be superconformal isomorphisms agreeing on overlaps. Morphisms of superschemes satisfy etale descent, and so there exists an isomorphism $\alpha: X \lgr Y$ of superschemes over $U$ such that $\pripull \alpha = \alpha_i$, and since $\alpha_i$ is superconformal the derivation
    \[ \alpha^*(\delta_Y) - \delta_X : \ox \to \omega_{X/U}\]
    vanishes along the \'etale cover $\pripull X \to X$ and thus vanishes globally, implying $\alpha^*(\delta_Y) = \delta_X$.
      \smallskip
    
      \noindent \emph{Descent for objects.} We are given a collection of stable super Riemann surfaces $\{(X_i/U_i, \delta_i) \}$ in $\ov{\mfr}(U_i)$ together with superconformal isomorphisms $\alpha_{ij}: \prijpull X_i  \lgr \prjipull X_j$ satisfying the cocycle condition on triple overlaps.
 
   By Proposition \ref{prop: stable supercurves are a superstack}, we know that the $X_i/U_i$ glue to a unique stable supercurve $X/U$ such that $\pripull X \cong X_i$, so we are left to show that $X$ admits a quasi superconformal structure $\delta$ satisfying  $\pripull \delta = \delta_i$.  Since the $\alpha_{ij}$ are superconformal,
   \[ \prijpull \delta_i =\alpha_{ij}^* \prjipull(\delta_j), \]
 and so the $\delta_i$ agree on all overlaps and thus glue to a global derivation, $\delta: \ox \to \omega_{X/U}$ such that $\pripull \delta= \delta_i$. 

We are left to check the two conditions for a derivation $\delta$ to be a qscf structure. For condition (i): First note that smoothness is an \'etale local property, so the smooth loci $Y_i \subset X_i$ glue to the smooth locus $Y \subset X$, and $\delta \vert_Y$ pulls back to $\delta_i \vert_{Y_i}$ for all $i$. The $\delta_i \vert_{Y_i}$ come from a superconformal structure $\dd_i$ defined on $Y_i \subset X$. The isomorphisms $\alpha_{ij}$ restrict to superconformal isomorphisms on all smooth overlaps, and thus the $\dd_i$ glue to a superconformal structure $\dd$ on $Y \subset X$ since the $\alpha_{ij}$ restrict to superconformal isomorphisms on all smooth overlaps. The superconformal structure $\dd$ determines a derivation $\delta': \mc{O}_Y \to \omega_{Y/U}$ which agrees with $\delta \vert_Y$ after pullback to the \'etale cover, and thus the two must agree on $Y$. 

For condition(ii): The condition is imposed on the geometric fibers of $X/U$ which are canonically identified with the geometric fibers of the $X_i/U_i$, and $\delta$ and $\delta_i$ both pullback to the same derivation on each geometric fiber. 
\end{proof}

\noindent \textbf{A1.}  
Let $X/T$ and $X/T$ be stable super Riemann surfaces over a base superscheme $T$. From Proposition \ref{proposition: rep of diagonal for prestable supercurves}, the isomorphism functor $\uniso_T(X,X')$ of the underlying stable supercurves is represented by an open subscheme $U$ of the Hilbert superscheme parmetrizing closed subschemes of $X \times_T X'$ which are flat and proper over $T$. 

It remains to identify the locus in $U$ of superconformal isomorphisms. Let $\phi: X_U \cong X_U'$ be the universal isomorphism over $U$, and let $\delta_{X_U}: \Omega_{X/U} \to \omega_{X/U}$ and $\delta_{X_U'}: \Omega_{X'/U} \to \omega_{X'/U}$ be the respective quasi superconformal structures. The $\phi_U$ induces isomorphisms
\[  \phi_U^* \Omega_{X'/U} \cong \Omega_{X/U}, \qquad \phi_U^* \omega_{X'/U} \cong \omega_{X/U}, \] 
and we can consider the morphism on $X_U$,
\[ \Delta: \Omega_{X/U} \to \omega_{X/U}, \] where $\Delta := \phi_{U}^*(\delta_{X_U'}) - \delta_{X_U}$. Since $\Delta$ is a a morphism of coherent sheaves on $X$, both flat over $U$, we can apply \cite[Lemma 6.1]{felder2020moduli}, to find that there exists a largest closed subscheme $Z \subset U$ such that $\Delta=0$ over $T$. 
Then $T$ is exactly the locus of superconformal isomorphisms. 
\vspace{10mm}

   \noindent \textbf{A2.} Since the stable super Riemann surface $X/A$ is finite type we can repeat the arguments from Proposition \ref{atwo: colimit preserving} to show there exists $\lambda$ such that $X\cong X_{\lambda} \times_{A_{\lambda}} A$ as superschemes over $A$. The quasi superconformal structure on $X_{\lambda}$ pulls back to a quasi superconformal structure on $X$, and this must agree with the quasi superconformal structure on $X$, since a stable supercurve admits at most one quasi superconformal structure. 
 \vspace{10mm}

      \noindent \textbf{A3.}   For (S1)a:  From Proposition \ref{prop: stable srs and quasi spin curves} it follows that the bosonic reduction of $\ovmfr$ is moduli space of quasi spin curves. This moduli space is an algebraic stack, \cite{jarvis1998torsion}, and so we can apply the induction argument from Proposition \ref{athree: schlessinger} to the outside groups in \eqref{ses for s srs}. Note that the subsheaf $\mc{A}_{X/T}^1 \subset \mct_{X/T}^1$ inherits the same formal properties since its defines as a kernel, so in particular, it is coherent and
      \[ \mca_{X/T} \otimes_A M \cong \mca_{X_0/T_0} \otimes_{A_0} M.\]
      That 
      For (S1)b, and (S2) we apply the same arguments to the outside groups of \eqref{ses for s srs}. 
      \vspace{10mm}

\noindent \textbf{A4.} By Proposition \ref{prop: effectivity for stable supercurves}, there exists a stable supercurve $X/T$ such that $\prjpull X \cong X_j$, so it remains to construct a qscf structure $\delta$ on $X/T$ satisfying $\prjpull \delta=\delta_j$. 
We use the super Grothendieck existence theorem from \cite{faroogh2019existence}:
\begin{equation} \label{GET}
\on{Coh}(X) \cong \on{Coh}(\mcx),
\end{equation}
where $\mcx$ is the formal superscheme determined by the projective system $(X_j, X_i \to X_j)$. The stable supercurve $X/T$ satisfies $\mcx \cong \varprojlim (\prjpull X)$ as formal superschemes over $\mct=\on{Spf}(R)$. 

It is convenient to work with the $\ox$-linear maps $\tildelt_j: \omegaxjtj \to \duaxjtj$ corresponding to $\delta_j$. Since $\omegaxjtj$ and $\duaxjtj$ are compatible with base change, they form projective systems of $\oxj$-modules with transition maps
\[
\ff_j \to \pr_{ij*}\prijpull \ff_j \cong \ff_i, \qquad \ff_j \in \{\omegaxjtj,\ \duaxjtj\},
\]
and the $\tildelt_j$ are compatible with these transitions. Thus they define a morphism of coherent systems. By \eqref{GET}, there exist unique coherent sheaves $\omegaxt$ and $\duaxt$ on $X$ with $\prjpull \omegaxt \cong \omegaxjtj$ and $\prjpull \duaxt \cong \duaxjtj$, together with a unique morphism
\[
\tildelt:\omegaxt \to \duaxt
\]
such that $\prjpull \tildelt=\tildelt_j$ for all $j$. Let $\delta:\ox \to \duaxt$ be the associated derivation.

We verify that $\delta$ is quasi superconformal. For (ii), since $p_j^*\delta=\delta_j$, the restriction of $\delta$ to each geometric fiber agrees with that of $\delta_j$, so the condition holds. 

For (i), let $\Sigma_j \subset X_j$ denote the singular loci. These are compatible under base change, so define a formal closed subscheme $\wh{\Sigma}\subset \mcx$ with ideal sheaf $\wh{\ii}$. By \eqref{GET}, this corresponds to a unique ideal sheaf $\ii\subset \ox$ defining a closed subscheme $\Sigma\subset X$, and we set $U=X\setminus \Sigma$. Then $\prjpull U \cong U_j$, so $U$ is smooth over $T$ since this holds on the central fiber.

On $U$, the restrictions $\delta_j|_{U_j}$ determine superconformal structures, which glue (by the same argument as in the descent step) to a superconformal structure $\dd_U$ on $U$, with associated derivation $\delta_U:\ou \to \omega_{U/T}$. Since both $\prjpull \delta_U$ and $\prjpull(\delta|_U)$ agree with $\delta_j|_{U_j}$, they define the same morphism of coherent systems on $\mcx$, hence $\delta_U=\delta|_U$. This verifies (i), and completes the proof.
\vspace{10mm}

\noindent \textbf{A5.} All obstruction modules are trivial, so coherent. 
\vspace{10mm}

\noindent \textbf{A6.} 
For (A6) apply the arguments from Proposition \ref{proposition: remaining conditions for stable supercurves} to the two outside groups of \eqref{ses for s srs}. Note both are coherent so we can apply Remark \ref{remark about freeness}. 
\vspace{10mm}

\noindent \textbf{A7, A8.}
For the remaining conditions, we can apply the arguments from Proposition \ref{prop: supercurves satisfy constructibility, condition asix} and Proposition \ref{prop: supercurves satisfy aeight}  to the two outside groups in \eqref{ses for s srs} directly as those arguments only used coherence.

\section{Appendix} \label{appendix}

\subsection{The $T^i$--functor formalism}

We describe the $T^i$--functor formalism in the super setting.
The construction is formally identical to the classical one
(cf.~\cite{hartshorne2010deformation}), but one must keep track of
the $\mathbb{Z}_2$--grading.

\paragraph{Truncated cotangent complex and definition of $T^i$.} Let $A \to B$ be a morphism of superrings, and let $M$ be a $B$-module. First choose a polynomial ring $R = A[x]$, where $x$ is a set of variables $\{x_i,\theta_i\}$ with $|x_i|=0$ and $|\theta_i|=1$ such that $R \to B$ is surjective. Let $I = \ker(R \to B)$. Next choose a free $R$-module $F$ together with a surjection $j \colon F \to I \to 0$, and let $Q = \ker(F \to I)$. There are no further choices. Let $F_0$ be the $R$-submodule of $F$ generated by the Koszul relations \[ j(a)b - a j(b) \] for all $a,b \in F$. Note that $F_0 \subset Q$. Indeed, \[ j(j(a)b - a j(b)) = j(a)j(b) - j(a)j(b) = 0 \] by $R$-linearity of $j$, and since $j(a), j(b) \in I \subset R$. The truncated cotangent complex is the complex of $B$-modules \[ L_2 \lra{d_2} L_1 \lra{d_1} L_0 \] defined as follows: \[ L_2 = Q/F_0, \qquad L_1 = F \otimes_R B = F/IF, \qquad L_0 = \Omega_{R/A} \otimes_R B. \] The map $d_2 \colon L_2 \to L_1$ is induced by the inclusion $i \colon Q \to F$. Indeed, $i(F_0) \subset IF$, since $j(a), j(b) \in I$ and $a,b \in F$, so \[ j(a)b - a j(b) \in IF. \] The map $d_1 \colon L_1 \to L_0$ is obtained by precomposing the map \[ d: I/I^2 \longrightarrow \Omega_{R/A} \otimes_R B \] in the conormal sequence with the map $F/IF \to I/I^2$ induced by $j$, noting that $j(IF) \subset I^2$. As in the classical case, $L_0$ and $L_1$ are free $B$-modules, since both are obtained by base change from free $R$-modules.

For any $B$--module $M$ we define
\[
T^i(B/A,M)=h^i\!\left(\underline{\on{Hom}}_B(L_\bullet,M)\right),
\]
where $\underline{\on{Hom}}$ denotes the internal Hom of graded
$B$--modules.

A point to keep in mind is the following.  The functors $T^i$
are defined using the \emph{internal} Hom in the category of
$B$--modules, not the grading--preserving Hom.  One basic reason for this is that we want 
\[
T^0(B/A,B) \cong T_{B/A}:= \underline{Der}_A(B,B).
\]

All classical properties of the $T^i$--functors described in
Chapter~1.1 of \cite{hartshorne2010deformation} continue to hold
verbatim in the super setting.

\paragraph{Deformation-obstruction theory for superschemes}
Let $f\colon X\to Y$ be a morphism of superschemes and let $\ff$
be a coherent sheaf on $X$.  The $T^i$--functors are compatible
with localization and therefore define coherent sheaves
\[
T^i(X/Y,\ff)
\]
on $X$.  For any affine open
\[
U=\Spec(B)\subset X
\]
with image $\Spec(A)\subset Y$, we set
\[
T^i(X/Y,\ff)(U)=T^i(B/A,M).
\]

For $i\in\{1,2\}$ the sheaves $\mct^i$ are supported along the
singular locus of $X$.

The following theorem describes the deformation--obstruction
theory of $X$.

\begin{theorem}[Compare to Theorem 10.2 in \cite{hartshorne2010deformation}]
\label{from Hart a}
Let $X$ be a superscheme flat and proper over $A$, and let
$A'\to A$ be a square--zero extension with kernel
$M=\ker(A'\to A)$ a finitely generated $A$--module. Then:
\begin{enumerate}
\item[(a)] There are three successive obstructions to the existence
of an extension $X'$ of $X$ over $A'$, lying in
\[
H^0(X,\mct_{X/A}^2 \otimes_A M)^+,\qquad
H^1(X,\mct_{X/A}^1\otimes_A M)^+,\qquad
H^2(X,\mct_{X/A} \otimes_A M))^+ .
\]

\item[(b)] If these obstructions vanish and a deformation $X'$
over $A'$ is fixed, there is an exact sequence of pointed sets
\[
0\to H^1(X,\mct_{X/A} \otimes_A M )^+ \to \Def(X/A,A') \to H^0(X,\mct_{X/A}^1 \otimes_A M)^+
\to H^2(X,\mct_{X/A} \otimes_A M)^+ .
\]
\end{enumerate}
\end{theorem}

\begin{proof}
The proof is identical to the classical case.
\end{proof}

If $M$ is finitely generated over the full reduction $A_0=A/\frak N_A$, then
\begin{equation}\label{reduce to the reduction}
\mct_{X/T}^i\otimes_A M
\cong
\mct_{X_0/A_0}^i\otimes_{A_0} M .
\end{equation}

For $i=0$ this follows from the identification $\mct_{X/T}^0 \cong \mct_{X/T}$ with the relative tangent sheaf. Indeed, for any affine open $U=\Spec(B)\subset X$, set $B_0=B/\frak N_A B$. Any $A$--derivation $\delta\colon B\to M$ satisfies $\delta(\frak N_A B)=0$, so it factors through $B\to B_0$, giving an $A_0$--linear derivation $\tilde\delta\colon B_0\to M$. This yields a canonical map
\[
\on{Der}_A(B,M)\to \on{Der}_{A_0}(B_0,M),
\]
whose inverse is obtained by composing an $A_0$--derivation $B_0\to M$ with the projection $B\to B_0$.

        \subsection{Computation of $T^1$ for the Ramond and NS models} \label{computation of ti for the nodes}

We compute the $T^1$-modules for both the Ramond and NS node. 

Let $A$ be a superring, and let $B$ be a finite-type $A$-superalgebra. 
The general procedure used to compute $T^1(B/A, B)$, or more generally $T^1(B/A,M)$, is the same as in the ordinary case: choose a surjection $R=A[x_1, \dots, x_n] \to B$, where $x_i$ can be even or odd, and denote its kernel by $I$. 
Then
\begin{equation} \label{tone}
T^1(B/A,B) = \on{coker}(\undhom(\Omega_{R/A} \otimes_R B,B) \xrightarrow{\psi} \homiisq)
\end{equation}
where $\psi$ is the $B$-linear morphism sending $\phi \in \undhom(\Omega_{R/A} \otimes_R B,B)$ to the composition $\phi \circ d$, and 
\[
d: \iisq \to \Omega_{R/A} \otimes_R B
\]
is the K\"ahler differential. 
\medskip
\paragraph{Ramond node.} Let $R=k[x,y,\theta]$, $I=(xy)$, and let $B=R/I$. 
\medskip

Choose a presentation $R \to I$, $r \mapsto r\cdot xy$, and compose with the canonical surjection $R \to R/I^2$. Then 
\[
\iisq=\on{image}\!\left(R \to R \to R/I^2\right).
\]
Since $I\cdot(I/I^2)=0$, the induced map $R \to I/I^2$, $r\mapsto \ov r\,xy$, factors through $B=R/I \to \iisq$. If $b\,xy\in I^2$, then $b\in I$, so this map is injective and hence induces a $B$-linear isomorphism $\iisq\cong B$. Therefore
\[
\homiisq\cong B.
\]
For $\phi\in\homomrk$,
\[
\psi(\phi)(xy)=(\phi\circ d)(xy)=x\,\phi(dy)+y\,\phi(dx).
\]
Since $\phi(dx)$ and $\phi(dy)$ are arbitrary elements of $B$, the image is $(x,y)$. Hence
\[
T^1(B/k,B)\cong B/(x,y)\cong k[\theta],
\]
a $k$-super vector space of dimension $(1|1)$.
\medskip

For a finite $B$-module $M$, the same argument with 
    \[
    \homomrk \Longrightarrow \undhom(\Omega_{R/k}\otimes B,M),
    \qquad
    \homiisq \Longrightarrow \undhom(\iisq,M),
    \]
    gives
    \[
    T^1(B/k,M)\cong M/(x,y)M=M\otimes_B k[\theta].
    \]
    \medskip

We can make the same computation over an arbitrary superring $A$: Let $R=A[x,y, \theta]$, $I=(xy-t)$, and let $B=R/I$. 
The same arguments show
\[ T^1(B/A,B) \cong B/(x,y)  \cong  B/(x,y, \theta) \oplus  B/(x,y, \theta) \theta.\] 

From the computation it follows that every deformation of the Ramond model over a square--zero extension is isomorphic to 
\[ \Spec A'[x,y, \theta]/(xy-a - \alpha \theta)\]
for unique $a \in (A')^+$ and $\alpha \in (A')\minus$ 
\medskip

\paragraph{NS case.} Let $R=\nspoly$, let $I=(xy, x \theta_2, y \theta_1, \theta_1 \theta_2)$, and let $B=R/I$. Note that $\Spec(B)$ is the standard \'etale neighborhood of an NS node on a stable supercurve. 
\medskip 

Choose a presentation $R^{2|2} \to I$,
\[
(r,s,t,u) \longmapsto r \cdot xy + s \cdot x \theta_2 + t \cdot y \theta_1 + u \cdot \theta_1 \theta_2,
\]
and compose with $R \to R/I^2$ to obtain 
\[
\iisq = \on{image}(R^{2|2} \to R \to R/I^2).
\]
Since $I \cdot (I/I^2)=0$, the induced map factors through $B^{2|2} \to \iisq$. Let $K$ denote its kernel, so $\iisq \cong B^{2|2}/K$, where $K$ is generated by
\[
k_1= (\theta_2,-y, 0,0), 
\qquad 
k_2=(\theta_1,0, -x, 0), 
\qquad 
k_3=(0, \theta_1,0,-x), 
\qquad 
k_4=(0, 0,\theta_2,y).
\]
\medskip

Let $L_1,L_2,L_3,L_4$ be a basis for $\undhom(B^{2|2},B)$. An element $L=fL_1+gL_2+hL_3+iL_4$ lies in $\homiisq$ if and only if $L\vert_K=0$, which gives
\[
f \theta_2 - gy =0, 
\qquad 
f \theta_1 - hx=0,
\qquad 
g\theta_1-ix=0, 
\qquad 
h\theta_2- i y=0.
\]
Hence
\[
f \in (x,y,\theta_1,\theta_2), 
\qquad 
g \in (x,\theta_1,\theta_2), 
\qquad 
h \in (y,\theta_1,\theta_2), 
\qquad 
i \in (\theta_1,\theta_2),
\]
and
\[
\homiisq =
(x,y,\theta_1,\theta_2)L_1
+ (x,\theta_1,\theta_2)L_2
+ (y,\theta_1,\theta_2)L_3
+ (\theta_1,\theta_2)L_4
\subset \undhom(B^{2|2},B).
\]

The image of $\psi$ in \eqref{tone} is generated by 
\[
(y,\theta_2,0,0), 
\qquad 
(x,0,\theta_1,0), 
\qquad 
(0,0,y,\theta_2), 
\qquad 
(0,x,0,\theta_1),
\]
and therefore
\begin{align*}
T^1(B/k,B) 
& = (x,y,\theta_1,\theta_2)/(x,y)\,L_1
+ (x,\theta_1,\theta_2)/(x,\theta_2)\,L_2  \\
& \quad 
+ (y,\theta_1,\theta_2)/(y,\theta_1)\,L_3
+ (\theta_1,\theta_2)/(\theta_1,\theta_2)\,L_4. \\
{} & \cong k^{2|2}
\end{align*}

We can make the same kind of computation over an arbitrary superring $A$: Let $R=A[x,y, \theta_1, \theta_2]$, $I=(xy-t^2, x \theta_2- t \theta_1, y \theta_1 -t \theta_2, \theta_1\theta_2)$, and let $B=R/I$. 
The same arguments show
\begin{align*}
     T^1(B/A,B) & \cong B/(x,y, \theta_1, \theta_2,t) \theta_1  \oplus  B/(x,y, \theta_1, \theta_2,t) \theta_2  \oplus  B/(x,y, \theta_1, \theta_2,t) \theta_1 \oplus B/(x,y, \theta_1, \theta_2,t)    \theta_2 \\
     {} & \cong  A/(t) \theta_1 \oplus A/(t) \theta_2 \oplus A/(t) \theta_1 \oplus A/(t) \theta_2
\end{align*} 

This shows that every deformation of the NS model over a square--zero extension $A' \to A$ is isomorphic to 
\[ \Spec A'[x,y, \theta_1, \theta_2]/(xy-t^2 - \alpha_1 \theta_1 - \alpha_2, x \theta_2 - t_1 \theta_1, y \theta_1 - t_2 \theta_2, 
\theta_1, \theta_2).\]

\subsection{\'Etale neighborhood of the node on a stable supercurve}

    \begin{lemma}[Compare with {\cite[Corollary 1.14]{alper2015artin}}]  \label{etalenbhd}
  Let $X_1,X_2$ be superschemes of finite type over $k$. Suppose $x_1 \in X_1, x_2 \in X_2$ are $k$-points such that $\wh{\mc{O}}_{X_1,x_1} \cong \wh{\mc{O}}_{X_2,x_2}$. There exists a common \'etale neighborhood 
\[
\begin{tikzcd}
            & {(X',x')} \arrow[ld] \arrow[rd] &             \\
{(X_1,x_1)} &                                 & {(X_2,x_2)}.
\end{tikzcd}
\]
In particular, if $x$ is a Ramond node on a stable supercurve $X$, then $(X,x)$ and 
\[ (Y,y)=(\Spec k[x,y, \theta]/(xy), x=y=\theta=0)\]
share a common \'etale neighborhood. 
If $x$ is a NS node, then the common \'etale neighborhood is
\[ (Y,y)= \Spec k[x,y, \theta_1, \theta_2]/(xy, x \theta_2, y \theta_1, \theta_1 \theta_2),x=y=\theta_1=\theta_2=0). \]

        \end{lemma}
        
        \begin{proof}
        For the first statement, we can use the same arguments as Corollary 1.14 of \cite{alper2015artin}, but with super Artin approximation \cite{ott2023artin}.
       The statements about the Ramond and NS nodes follow immediately from the definition.
        \end{proof}

        \subsection{Degree condition for strong projectivity}

        The next lemma is a statement about ordinary algebraic geometry. We include it for completeness. 

        \begin{lemma}\label{degree condition for sva} 
    Let $\pi\colon C \to T$ be a smooth proper family of curves of genus $g \ge 2$.  
    If $\mcl$ is a line bundle on $C$ such that $\deg(\mcl) \ge 2g+1$ on each connected component of $T$
    then $\mcl$ is strongly relatively very ample over $T$. 
    \end{lemma}

    \begin{proof}   
        We may assume $T$ is connected. Choose a point $t \in T$; by assumption $\deg(\mcl_t) \ge 2g+1$.
        \smallskip

        \noindent \emph{$\pist \mcl$ is locally free:} By Grauert’s theorem if
        \[
        t \longmapsto h^0(C_t,\mcl_t)
        \]
        is locally constant, then $\pist \mcl$ is locally free on an open subset of $t$. To show the above map is constant, it suffices to show that $h^1(C_t,\mcl_t)=0$ for all $t\in T$. Indeed, if this is the case, 
        \[
        h^0(C_t,\mcl_t)=\chi(\mcl_t), \qquad \forall t \in T
        \]
        and the Euler characteristic is constant in flat families. By Serre duality, and the assumption that $\deg(\mcl_t) \ge 2g+1$, 
        \[
        h^1(C_t,\mcl_t)=h^0(K_{C_t}-\mcl_t)=0
        \]
       since $\deg(K_{C_t}-\mcl_t)\le -3$. Thus $h^1(C_t,\mcl_t)=0$ for all $t$, as required.
        \smallskip

        \noindent \emph{$\pipull \pist \mcl \to \mcl$ is surjective:} Since thee canonical map $\pipull \pist \mcl \to \mcl$ is a morphism of locally free sheaves it is surjective if and only if its restriction to each fiber $C_t$ is surjective. Since $\deg(\mcl_t)\ge 2g+1$, the line bundle $\mcl_t$ is very ample on $C_t$,in particular, it is generated by its global sections, and so the map
        \begin{equation}\label{gen by global}
        H^0(C_t,\mcl_t)\otimes \mc O_{C_t}\;\longrightarrow\;\mcl_t
        \end{equation}
        is surjective.

   \noindent \emph{$C \to \pp_T(\pist \mcl)$ is a closed immersion:}. 
    For each $t \in T$, $\mcl_t$ is very ample on $C_t$, and so the induced map
    \[
C_t \to \pp(H^0(C_t, \mcl_t))
    \]
    is a closed immersion, and thus is a closed immersion globally. 
    \end{proof}

\printbibliography

\end{document}